\documentclass[12pt]{article}

\usepackage[english, main=english]{babel}
\usepackage{epsfig}
\usepackage{dsfont}
\usepackage{subfig}
\usepackage{amsmath,amssymb,amsthm,thmtools}
\usepackage{tikz}
\usepackage{xparse}
\usepackage{ifthen}
\usepackage{textcomp}
\usepackage{mathtools}
\usepackage{mathrsfs}
\usepackage{enumitem}
\usepackage{comment}
\usepackage[T1]{fontenc}
\usepackage[utf8]{inputenc}
\usepackage{graphicx}
\usepackage{pdfpages}
\usepackage{float}
\RequirePackage{times} 
\usepackage[colorlinks,linkcolor=blue,urlcolor=black]{hyperref}
\usepackage[dvipsnames]{xcolor}

\newcommand{\N}{\mathbb{N}} 
\newcommand{\Q}{\mathbb{Q}} 
\newcommand{\R}{\mathbb{R}} 
\newcommand{\C}{\mathbb{C}} 
\renewcommand{\P}{\mathbb{P}} 
\newcommand{\E}{\mathbb{E}} 
\newcommand{\1}{\mathds{1}}
\renewcommand{\Re}{\mathrm{Re}}

\renewcommand{\r}{\mathscr{R}}

\newcommand{\s}{\mathscr{S}}

\renewcommand{\d}{\mathrm{d}}

\theoremstyle{definition}
\newtheorem{definition}{Definition}[section]


\theoremstyle{plain}

\newtheorem{Lemma}[definition]{Lemma}
\newtheorem{Corollary}[definition]{Corollary}

\newtheorem{Theorem}[definition]{Theorem}

\theoremstyle{remark}
\newtheorem{Remark}{Remark}

\title{MA(1) processes with Laplace innovations conditioned to stay positive}
\author{Frank Aurzada and  Virginia Worf}

\begin{document}

\maketitle
	
\allowdisplaybreaks
\begin{abstract}
  We study a moving-average process with (not necessarily symmetric) Laplace innovations under the constraint of positivity. In the three nondegenerate parameter regimes $0<\theta<1$, $\theta>1$, and $\theta<0$, we prove convergence of the conditioned finite-dimensional distributions and identify the limit as a Doob $h$-transform. The regimes lead to qualitatively different limiting dynamics: the invariant law is supported on the positive half-line for $0<\theta<1$, the conditioned chain is confined to the negative half-line for $\theta>1$, and the dynamics are genuinely two-sided and governed by $q$-trigonometric functions for $\theta<0$. In each case, the persistence exponent, sharp persistence asymptotics, the defining eigenfunction, and the unique invariant distribution are obtained explicitly.
\end{abstract}
{\bf Keywords:} conditioned Markov chain; Doob's $h$-transform; invariant distribution; Laplace innovations; moving-average process; persistence probability; $q$-trigonometric function; $q$-ex\-po\-nen\-tial function;
\section{Introduction and main results}
Conditioning stochastic processes on events extending over long time horizons is a fundamental topic in probability theory. A classical example is Doob's construction of Brownian motion conditioned to remain positive \cite{doob}. Such conditioning procedures produce new stochastic models whose dynamics naturally differ from those of the original process. In this paper, we prove that conditioning a class of moving-average processes on positivity over increasingly long time intervals gives rise to a limiting process whose transition mechanism can be characterized through a Doob $h$-transform.

The model considered here provides an explicitly solvable instance of persistence conditioning for a Markov chain on a continuous, unbounded state space. Its main feature is a qualitative change in the conditioned dynamics as the parameter $\theta$ varies. In particular, the analytic form of the eigenfunctions, the effective state space, and the invariant distribution of the limiting chain depend essentially on $\theta$.
Nevertheless, in all three regimes $0<\theta<1$, $\theta>1$ and $\theta<0$, we determine explicitly the sharp persistence asymptotics, the Doob $h$-transform, and the unique invariant distribution. This complements the bounded uniform-innovation setting studied in \cite{AW}. The model may therefore serve as a tractable benchmark for persistence conditioning, since it remains explicitly solvable despite the qualitatively different limiting structures arising across the parameter regimes.
\subsection{Model and notation}
Let $(X_i)_{i\in\N_0}$ be i.i.d. random variables with an asymmetric Laplace distribution for some asymmetry parameter $\xi\in[0,1]$, i.e.\ with the density
\begin{align*}
   \phi_\xi( x)\coloneqq  \xi  e^x\mathds{1}_{x<0}+(1-\xi)e^{-x}\mathds{1}_{\{x>0\}},~~~~~~~~x\in\R.
\end{align*}
Whenever $\xi$ is fixed, we write $\phi=\phi_\xi$.
 For $\theta\in\R$, we define the process $(Z_i)_{i\in\N_0}$ as the MA(1) process
$Z_i\coloneqq X_{i+1}-\theta X_{i},~i\ge1$, and $Z_0\coloneqq X_1$.
The persistence probability of the process $(Z_i)_{i\in\N_0}$ can be written in terms of the Markov chain $M_0\coloneqq(X_1,0)$ and  $M_i\coloneqq(X_{i+1},X_i)$ for $i\geq 1$, as follows:
    \begin{align*}
   p_n^{\xi}(\theta)\coloneqq \P(\min_{1\leq i\leq n} Z_i> 0)=\P(X_2>\theta X_1,X_3>\theta X_2,\dots,X_{n+1}>\theta X_{n})=\P(\Omega_n),~~n\ge 1,
\end{align*}
where
$$
S\coloneqq \{(x_2,x_1)\in\R^2\mid x_2>\theta x_1\}\qquad\text{and}\qquad \Omega_n\coloneqq \cap_{i=1}^n\{M_i\in S\},
$$
and we set $p_0^{\xi}(\theta)\coloneqq1$. For $y\in\R$, let $\P_y$ denote the law of the chain with deterministic initial value $X_1=y$. The limiting finite-dimensional distributions of $(Z_i)_{i\in\N}$ conditioned to stay positive and started at $Z_0=y\in\R$ are represented by 
\begin{align}\label{def_lim_process}
    \lim_{n\to\infty}\P_y((M_k)_{1\leq k\leq m}\in\cdot\mid\Omega_n)=\P_y((M'_k)_{1\leq k\leq m}\in\cdot~),
\end{align}
provided that the process $M'$ on the right-hand side exists.
The transition probability of the Markov chain $M$ for $(x_1,x'_1),(x_2,x'_2)\in\R^2$ is given by
\begin{align*}
     \P_{(x_1,x'_1)}(M_1\in (\d x_2,\d x'_2))=\delta_{x_1}(\d x'_2)\P(X_1\in \d x_2)=\delta_{x_1}(\d x'_2)\phi(x_2)\d x_2.
\end{align*} 
Since we only consider two points satisfying $x'_2=x_1$, it suffices to consider $x_1, x_2\in \R$ instead of two $\R^2$-valued points. Therefore, we define the reduced transition kernel of $M$ by
   \begin{align*}
       p(x_1,\d x_2)\coloneqq \phi(x_2)\d x_2.
   \end{align*}
Consistent with the notation established in \cite{AR}, we use the following notation for $q>0$ and $n\in\N_0$. The $q$-Pochhammer symbol is defined as
\begin{align*}
    (a;q)_n\coloneqq\prod_{k=0}^{n-1} (1-a q^k)\quad\text{and, for }q<1, \quad(a;q)_\infty\coloneqq\prod_{k=0}^{\infty} (1-a q^k).
\end{align*}
For $q\neq 1$, the $q$-factorial is defined by
\begin{align*}
    [n]_q! \coloneqq \prod_{k=1}^{n}[k]_q   = \frac{(q;q)_n}{(1-q)^n},\quad\text{where}\quad [k]_q\coloneqq\frac{1-q^k}{1-q},
\end{align*}
with the continuous extension $[n]_1!=n!$. Moreover, we introduce the $q$-analogues of the classical exponential and trigonometric functions 
\begin{align*}
e_q(z)\coloneqq \sum_{n=0}^\infty \frac{z^{n}}{[n]_q!},\quad
    \cos_q(z) \coloneqq \sum_{n=0}^\infty \frac{(-1)^nz^{2n}}{[2n]_q!},
    \quad\text{and}\quad
    \sin_q(z) \coloneqq \sum_{n=0}^\infty \frac{(-1)^nz^{2n+1}}{[2n+1]_q!}.
\end{align*}
For $0<\theta<1$ and $y\in\R$, or for $\theta>1$ and $y\le 0$, we define
\begin{align*}
    \sin_\theta(z,y)&\coloneqq \sum_{j=0}^\infty\frac{(-1)^{j}z^{2j+1}}{[2j+1]_{\theta}!}e^{\theta\frac{1-\theta^{2j+1}}{1-\theta} y}, \quad\text{and}\quad\cos_\theta(z,y)\coloneqq\sum_{j=0}^\infty\frac{(-1)^{j}z^{2j}}{[2j]_{\theta}!}e^{\theta\frac{1-\theta^{2j}}{1-\theta} y},
\end{align*}
with the continuous extensions
$\sin_1(z,y)\coloneqq \sin(ze^y)$ and $\cos_1(z,y)\coloneqq\cos(ze^y)$. For the properties of these functions, we refer to Section~\ref{sec:auxiliaryqexpetc}.
 For $\theta>0$, denote by $z_c=z_c(\theta)$ the smallest positive real zero of $\cos_{\theta}$.  By Lemma 12 in \cite{AR} for $0<\theta\le 1$ and Corollary~\ref{arctan_cor} below for $\theta>1$, 
\begin{align*}
    \tan_{\theta}:[0,z_c)\to[0,\infty),\quad z\mapsto\frac{\sin_\theta(z)}{\cos_\theta(z)}
\end{align*}
is a bijective, strictly increasing function, so we may define its inverse $\arctan_{\theta}$. 

\subsection{Main results}
We now state our main results. The limiting transition kernel is determined by an eigenvalue $\lambda$ and an eigenfunction $h$ of the
sub-stochastic kernel of the Markov chain $M$ killed upon leaving $S$. The
corresponding eigenvalue equation is
\begin{align}\label{eq_eigenfunction}
    \lambda h(x)=\int_{\{y>\theta x\}}\!\!\!\!\!\!\!h(y)\phi(y)\d y=\int_{\{y>\theta x\}\cap(-\infty,0)}\!\!\!\!\!\!\!\xi  e^yh(y)\d y+\int_{\{y>\theta x\}\cap(0,\infty)}\!\!\!\!\!\!\!(1-\xi)  e^{-y}h(y)\d y.
\end{align}
The form of $\lambda$ and $h$, and consequently that of the limiting process, depends on
the parameter regime, leading to separate expressions for $0<\theta<1$,
$\theta>1$, and $\theta<0$. Theorem~\ref{maintheorem} identifies the
limiting process and its transition kernel, while
Theorem~\ref{invariant_distribution} determines its unique invariant
distribution.
\begin{Theorem}\label{maintheorem}
    Let $\theta\in\R$ with $\theta\neq1$, $\xi\in[0,1]$, and
$\Omega\in\{\R,\R_{<0}\}$, and assume that
$(\theta,\xi,\Omega)$ belongs to one of the parameter regimes
 (i)-(v) below. Then, for every $y\in\Omega$, the limiting process $M'$ in \eqref{def_lim_process} exists and is a time-homogeneous Markov chain on $\Omega$ with reduced transition kernel 
\begin{align*}
     p'(x_1,\d x_2)&\coloneqq\mathds{1}_S(x_2,x_1)\frac{h(x_2)}{h(x_1)}\frac{1}{\lambda}\phi(x_2)\d x_2,\quad x_1,x_2\in\Omega,
\end{align*}
 where $\lambda$ and $h:\R\to\R$ satisfy \eqref{eq_eigenfunction} and are given as follows:
\begin{enumerate}[label=(\roman*)]
     \item  Let $0<\theta<1$, $\xi\in [0,1)$ and $\Omega=\R$. Then, $\lambda=(1-\xi)(1-\theta)$ and
     \begin{align*}
         h(x)=\begin{cases}
          e^{\frac{\theta}{\theta-1}x},~~~&x>0,
          \\e^{\frac{\theta x}{1-\theta}}\sum_{m=0}^\infty\frac{\left(-\frac{\theta x}{1-\theta}\right)^m}{m!}\left(\frac{-\xi}{\lambda}(1-\theta);\theta\right)_m,~~~~&x\le 0.
      \end{cases}
     \end{align*}
     \item  Let $\theta>1$, $\xi\in(0,1]$ and $\Omega=\R_{<0}$. Then, $\lambda=\left(1-\frac{1}{\theta}\right)\xi$ and
     \begin{align*}
         h(x)= \begin{cases}
          0,~~~&x>0,
          \\{e_\theta\left(\frac{1-\xi}{\lambda}\right)}\sum_{j=0}^\infty \frac{\left(\frac{-\xi}{\lambda}\right)^{j}}{[j]_\theta!}e^{\frac{\theta(1-\theta^j)}{1-\theta}x},\quad &x\le0.\end{cases}
     \end{align*}
 \item Let $\theta <0$, $\xi\in(0,1)$ and $\Omega=\R$. Then, $\lambda= \frac{\sqrt{\xi(1-\xi)}}{\arctan_{-\theta}\left(\sqrt{\frac{\xi}{1-\xi}}\right)}$ and
 \begin{align*}
      h(x)=\begin{cases}
       \cos_{-\theta}\left(\arctan_{-\theta}\left(\sqrt{\frac{\xi}{1-\xi}}\right),-x\right) ,~~&x>0,
        \\\frac{\sqrt{1-\xi}}{\sqrt{\xi}}\sin_{-\theta}\left(\arctan_{-\theta}\left(\sqrt{\frac{\xi}{1-\xi}}\right),x\right),~~&x\le 0.
        \end{cases}
 \end{align*}
 \item Let $\theta <0$, $\xi=0$ and $\Omega=\R$. Then, $\lambda=1$ and \begin{align} \label{eqn:lasth}
    h(x)=\begin{cases}
        1,~~&x\ge 0,
        \\ e^{-\theta x},&x<0.
    \end{cases}
\end{align}
 \item Let $\theta=0$, $\xi\in[0,1)$ and $\Omega=\R$. Then, $\lambda=1-\xi$ and $h(x)\equiv 1$.
\end{enumerate}
\end{Theorem}
    The case $\theta=1$ is excluded in Theorem~\ref{maintheorem} since the conditioning leads to a degenerate limit, as $$\P(y<X_2<X_3<\dots<X_{n+1})=\frac{\P(X_1>y)^n}{n!}.$$  
    The persistence behavior in the parameter cases $\xi\in[0,1]$
    that are not covered by Theorem~\ref{maintheorem} does not fall within the positive-exponential-rate Doob-transform framework and is discussed separately in Section~\ref{trivialcases}.

    In each of the cases of Theorem~\ref{maintheorem}, $\lambda$ is  
    the persistence exponent of the persistence probabilities, corresponding to the decay rate defined by (cf.\ Section~2.1 in \cite{AMZ})
    $$
    \log(\lambda)\coloneqq\lim_{n\to\infty}\frac{1}{n}\log p_n^\xi(\theta).
    $$

We now turn to the invariant distribution of $p'$.
\begin{Theorem}\label{invariant_distribution}
  Let $(E_k)_{k\ge1}$ be i.i.d. exponentially distributed random variables with parameter $1$.    Let $\theta\in\R$ with $\theta\neq1$, $\xi\in[0,1]$, and
$\Omega\in\{\R,\R_{<0}\}$, and assume that
$(\theta,\xi,\Omega)$ belongs to one of the parameter regimes
 (i)-(v) below. Then, the reduced  kernel $p'$ of $M'$ from Theorem~\ref{maintheorem} admits a unique invariant probability measure $\pi$ on $\Omega$.
Explicitly, $\pi$ is given as follows:
    \begin{enumerate}[label=(\roman*)]
     \item  Let $0<\theta<1$, $\xi\in[0,1)$ and $\Omega=\R$. Then, 
     \begin{align}\label{inv_distr_i}
         \pi(\d x) 
         &= \P\left( (1-\theta)\sum_{k=0}^\infty\theta^kE_{k+1} \in\d x \right).
     \end{align}
     \item  Let $\theta>1$, $\xi\in(0,1]$ and $\Omega=\R_{<0}$. Then, 
     \begin{align*}
        \pi(\d x) 
       &= \P\left( -(1-\theta^{-1}) \sum_{k=0}^\infty\theta^{-k}E_{k+1} \in\d x \right).
     \end{align*}
 \item Let $\theta <0$, $\xi\in(0,1)$ and $\Omega=\R$. Then, with $a\coloneqq \arctan_{-\theta}\left(\sqrt{\frac{\xi}{1-\xi}}\right)$,
\begin{align*}
    &\pi(\d x) =\frac{1}{b}
    \begin{cases} e^{-x} \cos_{-\theta}\left(a,-x\right) \cos_{- \theta^{-1}}\left(a,-x\right)\d x, &x>0,\\ e^x \sin_{-\theta}\left(a,x\right) \sin_{- \theta^{-1}}\left(a,x\right)\d x, &x\leq0, \end{cases}
\end{align*}
where $b\!\coloneqq \!\int_0^\infty e^{-x}\! \cos_{-\theta}(a,-x) \cos_{-\theta^{-1}}(a,-x)\d x + \int_{-\infty}^0 e^x \!\sin_{-\theta}(a,x) \sin_{-\theta^{-1}}(a,x)\d x$.
 \item Let $\theta <0$, $\xi=0$ and $\Omega=\R$. Then, 
  $\pi(\d x) =\P(E_1\in\d x)$.
 \item Let $\theta=0$, $\xi\in[0,1)$ and $\Omega=\R$. Then, 
  $\pi(\d x) =\P(E_1\in\d x)$.
\end{enumerate}
\end{Theorem}
    We would like to stress in particular the interesting qualitatively different ranges: For $0<\theta<1$, the conditioned chain becomes eventually positive and has an AR(1)-type stationary law. For $\theta>1$, the Doob limit is confined to the negative half-line and also has an AR(1)-type stationary law. For $\theta<0$, the dynamics remain genuinely two-sided with an exotic explicit $q$-trigonometric invariant density.
    We also remark that in each case except for $(iii)$ the invariant distribution is independent of $\xi$.

For the remaining part of the paper, we will additionally work with the persistence probability $ p_n^{\xi,y}(\theta)$ started at some $y\in\R$, defined by $p_0^{\xi,y}(\theta)\coloneqq 1$ and
\begin{align*}
    p_n^{\xi,y}(\theta)\coloneqq \P(X_2>\theta y,X_3>\theta X_2,\dots,X_{n+1}>\theta X_{n})=\P_y(\Omega_n),~~~~n\ge 1.
\end{align*}
The generating function of the persistence probabilities  $(p_n^{\xi,y}(\theta))$ is defined by $$\hat{P}_{\xi,y,\theta}(z)\coloneqq \sum_{n=0}^\infty p_n^{\xi,y}(\theta)z^n.$$
\paragraph{Related work.}
Our starting point is \cite{AR}, where explicit formulas are derived for the persistence probabilities of moving-average processes with Laplace innovations started from a fixed point. While \cite{AR} focuses on these persistence probabilities, the present work studies the stochastic process selected by long-time positivity conditioning. Here, we study the conditioned process, identify its Doob $h$-transform and invariant distribution, and derive sharp persistence asymptotics following the approach developed in \cite{AW}. A general framework for non-exit probabilities of Markov chains, including results for moving-average processes with uniform innovations, is developed in \cite{AMZ}.

The Doob $h$-transform originates in \cite{doob}; general accounts are given in \cite{RogersWilliams1,RogersWilliams2}, while related results on long-time conditioning and quasi-stationarity include \cite{ChampagnatVillemonais2016,ColletMartinezSanMartin,SenetaVereJones}. An easily accessible countable-state formulation appears in \cite[Theorem 2.11]{Swart}. Closely related transforms for random walks under positivity or survival constraints are studied in \cite{BertoinDoney,Caravenna2005,CaravennaChaumont2008,CaravennaChaumont2013,Wachtel1,Wachtel4,Wachtel6,Wachtel7,Wachtel5,Wachtel2,Wachtel3,EichelsbacherKonig}.

Persistence for moving-average processes has also been studied for uniform innovations \cite{AurzadaRaschel} and Gaussian innovations \cite{2,12}; an early contribution concerning stationary sequences of MA(1) type is \cite{15}. Related autoregressive models are treated in \cite{1,3,5,9,DonnartSimon,11,14}. For the broader statistical-physics perspective on persistence, see \cite{8}.

\paragraph{Outline.} Section~\ref{sec_preliminaries} collects probabilistic preliminaries on conditioned Markov chains and analytic properties of the $q$-exponential and $q$-trigonometric functions. Section~\ref{sec_theta_pos} treats the case $\theta>0$. We first derive generating-function representations valid for $\theta\ne1$ and then analyze the regimes $0<\theta<1$ and $\theta>1$ separately to obtain the persistence asymptotics, the corresponding Doob $h$-transforms, and the invariant distributions of the conditioned chains. Section~\ref{proofs_theta_neg} considers the case $\theta<0$, where $q$-trigonometric functions enter the spectral analysis, and proves parts (iii) of Theorems~\ref{maintheorem} and \ref{invariant_distribution}. Finally, Section~\ref{trivialcases} discusses parts (iv) and (v) as well as the remaining boundary cases not covered by Theorem~\ref{maintheorem}.

\section{Preliminaries}\label{sec_preliminaries}
\subsection{Auxiliary results on conditional Markov chains}
We start with a uniform upper bound on the persistence probabilities useful for applying the dominated convergence theorem in some instances below. The proof is straightforward. 

\begin{Lemma}\label{estimate_Px}
For $n\in\N$, it holds 
$p_n^{\xi,y}(\theta)\le p_{n-1}^{\xi}(\theta)$.
\end{Lemma}
We continue with two lemmas that prove the eigenvalue equation \eqref{eq_eigenfunction} and the Doob $h$-transform of the reduced transition kernel when the asymptotics of the persistence probabilities is known.
\begin{Lemma}\label{eigenfunction_lemma}
  Assume that there is a continuous, non-negative function $h$ and constants $\lambda=\lambda(\theta,\xi)>0$ and $c(\theta,\xi),c_2(\theta,\xi)>0$ such that for  every $x\in\R$,
    \begin{align} \label{eqn:ppconvergence1and2}
        \lim_{n\to\infty}\frac{p_n^{\xi,x}(\theta)}{c(\theta,\xi)\lambda^{n+1}}= h(x)
   \quad\text{and}\quad
       \lim_{n\to\infty}\frac{p_n^{\xi}(\theta)}{c_2(\theta,\xi)\lambda^{n+1}}=1.
    \end{align}
 Then, for all $x\in\R$,
 \begin{align*}
     \lambda h(x)=\int_{\{y>\theta x\}}h(y)\phi(y)\d y.
 \end{align*}
\end{Lemma}
\begin{proof}
Using the Markov property,
  \begin{align*}
   \lambda h(x)&= \lim_{n\to\infty}\frac{p_n^{\xi,x}(\theta)}{c(\theta,\xi)\lambda ^n}
      =\lim_{n\to\infty}\frac{\P_x(M_1\in S,\dots,M_n\in S)}{c(\theta,\xi)\lambda ^n}
      \\&=\lim_{n\to\infty}\int_\R\frac{\P_y(M_1\in S,\dots,M_{n-1}\in S)}{c(\theta,\xi)\lambda ^n}\1_{\{(y,x)\in S\}}p(x,\d y)
      \\&=\lim_{n\to\infty}\int_{\{y>\theta x\}}\frac{\P_y(M_1\in S,\dots,M_{n-1}\in S)}{c(\theta,\xi)\lambda ^n}p(x,\d y).
    \end{align*}
    Using Lemma~\ref{estimate_Px} and both assumptions in (\ref{eqn:ppconvergence1and2}), one can find an integrable majorant and apply the dominated convergence theorem, so that the last expression equals
\begin{align*}
     &
     \int_{\{y>\theta x\}}\lim_{n\to\infty}\frac{\P_y(M_1\in S,\dots,M_{n-1}\in S)}{c(\theta,\xi)\lambda ^n}p(x,\d y)
      =\!\int_{\{y>\theta x\}} h(y)\phi(y)\d y.\tag*{\qedhere}
\end{align*}
\end{proof}
We next show that  asymptotics of the persistence probabilities yield convergence to the corresponding Doob $h$-transform.
\begin{Lemma}\label{Kernel_lim_process}
 Let $\theta\in\R$, $\xi\in[0,1]$. Assume that there is some non-negative, continuous function $h$ such that for all $x\in\R$, it holds
\begin{align}\label{assumption_asymptotics}
   \lim_{n\to\infty}\frac{p_n^{\xi,x}(\theta)}{c(\theta,\xi)\lambda^{n+1}}= h(x) \qquad\text{and}\qquad
        \lim_{n\to\infty}\frac{p_n^{\xi}(\theta)}{c_2(\theta,\xi)\lambda^{n+1}}=1
\end{align}
for some constants $\lambda=\lambda(\theta,\xi)>0$ and $c(\theta,\xi),c_2(\theta,\xi)>0$.
Moreover, let $\Omega=\R$ if $\theta\le 1$ and $\Omega=\R_{<0}$ if $\theta>1$ and assume that $h(x)>0$ for all $x\in\Omega$ and $h(x)=0$ for $x\notin\Omega$. Then, for $y\in\Omega$, a limiting process $M'$ defined in \eqref{def_lim_process}
exists, and $M'$ is a Markov chain on $\Omega$ with reduced transition kernel 
\begin{align} \label{eqn:tkrefk}
    p'(x_1,\d x_2)\coloneqq \mathds{1}_S(x_2,x_{1})\frac{h(x_2)}{h(x_1)}\frac{1}{\lambda}\phi(x_2)\d x_2,~~~~~~~~x_1,x_2\in \Omega.
\end{align}
\end{Lemma}
\begin{proof} Recall that $M_i=(X_{i+1},X_i)$ and that we denote the reduced transition kernel of the Markov chain $(M_i)_{i\in\N}$ by $p(x_1,\d x_2)=\phi(x_2)\d x_2$ for $x_1, x_2\in\R$. Let $C_1,\dots, C_m\in\mathcal{B}(\R)$ and $\mathbf{x}\coloneqq(x_1,\dots,x_m)\in \R^m$ with $x_0\coloneqq y\in\Omega$. For fixed $n\ge m$, the Markov property yields
\begin{align*}
     &\lim_{n\to\infty}\P_{y}((M_k)_{1\leq k\leq m}\in (C_1\times\{y\})\times\dots\times (C_m\times C_{m-1})\mid \Omega_n)   \\&=\lim_{n\to\infty}\int_{\R^m}\1_{C_1\times\dots\times C_m}(\mathbf{x})\left[\prod_{i=1}^m p(x_{i-1},\d x_i)\mathds{1}_S(x_i,x_{i-1})\right]\frac{\P_{x_m}(\Omega_{n-m})}{\P_{x_0}(\Omega_n)} .
\end{align*}
Using Lemma~\ref{estimate_Px} and both assumptions in (\ref{assumption_asymptotics}), we can find an integrable majorant and apply the dominated convergence theorem, so that the last expression equals
\begin{align}
    &
    \int_{\R^m}\1_{C_1\times\dots\times C_m}(\mathbf{x})\left[\prod_{i=1}^m p(x_{i-1},\d x_i)\mathds{1}_S(x_i,x_{i-1})\right]\lim_{n\to\infty}\frac{\P_{x_m}(\Omega_{n-m})}{\P_{x_0}(\Omega_n)}  \notag
    \\&=\int_{\R^m}\1_{C_1\times\dots\times C_m}(\mathbf{x})\left[\prod_{i=1}^m p(x_{i-1},\d x_i)\mathds{1}_S(x_i,x_{i-1})\right]\frac{h(x_m)}{h(x_0)\lambda^m}.\notag
\end{align}
Note that there is only one case where $\Omega\ne \R$, which is $\theta>1$ where $\Omega=\R_{<0}$. Let us examine this case. If there is some $1\le i\le m$ with  $\mathds{1}_S(x_i,x_{i-1})=0$, the integrand is zero. Assume that $\mathds{1}_S(x_i,x_{i-1})=1$ for all $1\le i\le m$. We started with $x_0\in\Omega$, so $h(x_0)>0$, and if the path leaves $\Omega$, define $j\coloneqq\min\{1\leq i\leq m:x_i\notin\Omega\}$. Then, we have $x_{k}\ge0$ for all $j\le k\le m$, so the rest of the path $x_{j},\dots x_m$ is outside of $\Omega$ as well. In particular, $h(x_m)=0$, which coincides with the fact that if a path leaves $\Omega$, it stays outside of $\Omega$ and has zero mass in the limiting distribution. Therefore, even if $\Omega\neq \R$, we can restrict the integration domain above to $\Omega^m$ and get
\begin{align}
    &\lim_{n\to\infty}\P_{y}((M_k)_{1\leq k\leq m}\in (C_1\times\{y\})\times\dots\times (C_m\times C_{m-1})\mid \Omega_n)\notag
    \\&=\int_{\Omega^m}\1_{C_1\times\dots\times C_m}(\mathbf{x})\left[\prod_{i=1}^m p(x_{i-1},\d x_i)\mathds{1}_S(x_i,x_{i-1})\frac{h(x_i)}{h(x_{i-1})}\frac{1}{\lambda}\right].\notag
\end{align}
Therefore, the reduced transition kernel of the limiting process $M'$ is given by (\ref{eqn:tkrefk}).
The fact that $\int_\Omega p'(x_1,\d x_2) =1$ follows from \eqref{eq_eigenfunction}, by Lemma~\ref{eigenfunction_lemma}.
\end{proof}

Using the transition kernel from the last lemma, we can now identify the invariant distribution.
\begin{Lemma}\label{invariant_distribution_limiting_process} 
Let $\Omega\in\{\R,\R_{<0}\}$, $\theta\in\R$, $\lambda\in\R_{>0}$ and $\xi\in[0,1]$ and assume that there is some measurable, non-negative function $h$ that is positive on $\Omega$ such that
\begin{align*}
    p'(x_1,\d x_2)\coloneqq \mathds{1}_S(x_2,x_{1})\frac{h(x_2)}{h(x_1)}\frac{1}{\lambda}\phi(x_2)\d x_2,~~~~~~~~x_1,x_2\in \Omega
\end{align*}
is a Markov kernel. Further, assume that there exists a measurable function $\tilde h:\Omega\to[0,\infty)$ satisfying the left eigenfunction equation \begin{align}\label{eq_left_eigenfunction} 
\lambda \tilde h(y)=\int_{\Omega\cap\{y>\theta x\}}\tilde h(x)\phi(x)\d x,\qquad y\in\Omega,
\end{align} 
and
\begin{align*} 
0< I\coloneqq \int_\Omega h(x)\tilde h(x)\phi(x)\d x <\infty. 
\end{align*} 
Then the probability measure 
\begin{align*} 
\pi(\d x) \coloneqq \frac{h(x)\tilde h(x)}{I} \phi(x)\d x,\qquad x\in\Omega,
\end{align*} 
is invariant for $p'$, that is, for Borel sets $B\subseteq\Omega$
\begin{align*} 
\int_\Omega \pi(\d x_1)p'(x_1,B) = \pi(B). 
\end{align*} 
\end{Lemma} 
\begin{proof} 
By Tonelli's theorem and the definition of $p'$, we obtain 
\begin{align*} 
\int_\Omega \pi(\d x_1)p'(x_1,B) 
&= \int_\Omega \frac{h(x_1)\tilde h(x_1)}{I} \left( \int_B \mathds{1}_S(x_2,x_1) \frac{h(x_2)}{\lambda h(x_1)} \phi(x_2)\d x_2 \right) \phi(x_1)\d x_1 
\\ &= \frac{1}{I\lambda} \int_B h(x_2) \left( \int_\Omega \mathds{1}_{\{x_2>\theta x_1\}} \tilde h(x_1) \phi(x_1)\d x_1 \right) \phi(x_2)\d x_2. 
\end{align*} 
Using the left eigenfunction equation \eqref{eq_left_eigenfunction}, this equals
\begin{align*} 
& \frac{1}{I\lambda} \int_B h(x_2)\lambda\tilde h(x_2) \phi(x_2)\d x_2
= \frac{1}{I} \int_B h(x_2)\tilde h(x_2) \phi(x_2)\d x_2 
= \pi(B). \tag*{\qedhere}
\end{align*} 
\end{proof}
The last lemma of this section connects the representation of a power series to the asymptotic behavior of its coefficients. The proof is standard.
\begin{Lemma}\label{asymptotic_powerseries}
    Let $(p_n)$ be a sequence of real numbers. Let $z_0>0$. Assume that
\begin{equation}
\sum_{n=0}^\infty p_n z^n = \frac{f(z)}{z_0-z},
\end{equation}
for complex $|z| < z_0$. Further assume that $f$ is analytic on $|z|<z_1$ with $z_0<z_1$ and that $f(z_0)\neq 0$. Then
$$
p_n \sim \frac{f(z_0)}{z_0} \cdot z_0^{-n},\qquad \text{as $n\to\infty$.}
$$
\end{Lemma}
\subsection{Auxiliary results on q-exponential and q-trigonometric functions} \label{sec:auxiliaryqexpetc}
We next collect the analytic properties of the $q$-exponential and $q$-trigonometric functions needed to locate the singularities of the generating functions. Let us start with a representation of $e_\theta$ for $\theta\in(0,1)$.
\begin{Lemma}\label{e_theta_singularities_thetale1}
   For $0<\theta< 1$, the power series defining $e_\theta$ has
radius of convergence $1/(1-\theta)$. For all $z\in\C$ outside its poles, the meromorphic continuation is given by
   \begin{align}\label{e_theta_prod_thetale1}
       e_\theta(z)=\frac{1}{(z(1-\theta);\theta)_\infty}=\prod_{n=0}^\infty\frac{1}{1-z(1-\theta)\theta^n}.
   \end{align}
   Moreover,  the function $z\mapsto e_\theta(z)$ has no zeros and has simple poles precisely at $\left(\frac{1}{(1-\theta)\theta^n}\right)_{n\in\N_0}$. In particular, its poles are real, positive, simple and isolated. The function is strictly positive on the real interval $\left(-\infty,\frac{1}{1-\theta}\right)$.
\end{Lemma}
\begin{proof}
The proof follows from the representation
    \begin{align*}
        e_{\theta}(z)
        &= \sum_{n=0}^\infty \frac{z^{n}}{[n]_{\theta}!}=\sum_{n=0}^\infty\frac{z^n(1-\theta)^n}{(\theta;\theta)_n}=\frac{1}{(z(1-\theta);\theta)_\infty}=\prod_{n=0}^\infty\frac{1}{1-z(1-\theta)\theta^n},
    \end{align*}
    which is due to the $q$-binomial theorem and is valid for all $|z|<1/(1-\theta)$.
\end{proof}
In the following, set $\bar\theta\coloneqq\theta^{-1}$. 
\begin{Lemma}\label{e_theta_entire_thetage1}
    For $\theta>1$, the function $z\mapsto e_\theta(z)$ is entire and, for all $z\in\C$, it holds
    \begin{align}\label{e_theta_prod_thetage1}
        e_\theta(z)=\sum_{n=0}^\infty\frac{\bar\theta^{\frac{n(n-1)}{2}}((1-\bar\theta)z)^n}{(\bar\theta;\bar\theta)_n}=(-z(1-\bar\theta);\bar\theta)_\infty=\prod_{k=0}^\infty(1+z(1-\bar\theta)\bar\theta^k),
    \end{align}
    and its zeros are given by the sequence $\left(\frac{-1}{(1-\bar\theta)\bar\theta^n}\right)_{n\in\N_0}$.  In particular, its zeros are real, negative, simple and isolated.
\end{Lemma}
\begin{proof}
We have
\begin{align*}
    [n]_\theta!
    =\prod_{k=1}^n\theta^{k-1}\frac{1-\bar\theta^k}{1-\bar\theta}
    =\prod_{k=1}^n\theta^{k-1}[k]_{\bar\theta}
    =\theta^{\frac{n(n-1)}{2}}[n]_{\bar\theta}!
    =\theta^{\frac{n(n-1)}{2}}\frac{(\bar\theta;\bar\theta)_n}{(1-\bar\theta)^n},
\end{align*}
so that
\begin{align*}
    e_{\theta}(z)
        &= \sum_{n=0}^\infty \frac{z^{n}}{[n]_{\theta}!}
        =\sum_{n=0}^\infty\frac{\bar\theta^{\frac{n(n-1)}{2}}((1-\bar\theta)z)^n}{(\bar\theta;\bar\theta)_n}.
\end{align*}
The right-hand side is an entire function in $z$ since $\lim_{n\to\infty}(\bar\theta;\bar\theta)_n=(\bar\theta;\bar\theta)_\infty>0$ and $\bar\theta<1$.
By Cauchy-Hadamard, the power series thus has radius of convergence $R=\infty$. Moreover, using the identity $e_\theta(z)=\frac{1}{e_{\bar\theta}(-z)}$ and Lemma~\ref{e_theta_singularities_thetale1}, we obtain (\ref{e_theta_prod_thetage1}).
\end{proof}
The following comparison will be used to exclude non-real zeros of the relevant denominators.
\begin{Lemma}\label{estimate_e_theta_ge_1_prelim}
    For $\theta>0$ and $z\in\C$ outside the poles in case $0<\theta<1$, it holds
    \begin{align*}
        |e_\theta(z)|&>|e_\theta(-z)|,~~~\Re(z)>0,
        \qquad \text{and}\qquad
        |e_\theta(z)|<|e_\theta(-z)|,~~~\Re(z)<0.
    \end{align*}
\end{Lemma}
\begin{proof}
For $\theta=1$, the result follows from known properties of the standard exponential function, so consider $\theta\ne 1$.
  Let $z=x+iy$. 
  First, consider $\theta >1$. We use the product representation from \eqref{e_theta_prod_thetage1} to obtain
  \begin{align*}
       |e_{\theta}(\pm z)|^2&=\prod_{k=0}^\infty((1 \pm x(1-\bar\theta)\bar\theta^k)^2+(y(1-\bar\theta)\bar\theta^k)^2).
  \end{align*}
  Now, for each $k\in\N_0$, we have
  \begin{align*}
      (1+x(1-\bar\theta)\bar\theta^k)^2+(y(1-\bar\theta)\bar\theta^k)^2-((1-x(1-\bar\theta)\bar\theta^k)^2+(y(1-\bar\theta)\bar\theta^k)^2)=4x(1-\bar\theta)\bar\theta^k,
  \end{align*}
  which is strictly positive for $x>0$ and strictly negative for $x<0$. For $0<\theta<1$, using the identity $e_\theta(z)=\frac{1}{e_{\bar\theta}(-z)}$ and the assertion for $\theta>1$, the claim follows. 
\end{proof}
\begin{Corollary}\label{estimate_pos_e_prime_e}
    For $\theta>0$, $\theta\ne 1$ and $z\in\R_{>0}$, the quantity $\frac{e'_{\theta} (iz)}{e_{\theta} (iz)}+\frac{e'_{\theta}(-iz)}{e_{\theta}(-iz)}$ is real and positive.
\end{Corollary}
\begin{proof}
The quotients are well-defined since $e_\theta$ is holomorphic and non-zero in $\pm iz$: For $\theta>1$, by Lemma~\ref{e_theta_entire_thetage1}, $e_\theta$ is entire and its zeros are real, so $iz\notin \R$ is not a zero. For $0<\theta<1$, by Lemma~\ref{e_theta_singularities_thetale1}, $e_\theta$ is meromorphic with no zeros and only real singularities, so $iz\notin\R$ is neither a zero nor a singularity. It remains to show that the quantity is real and positive.
    First, let $\theta >1$.
    By \eqref{e_theta_prod_thetage1}, $ e_\theta(w)=\prod_{k=0}^\infty(1+(1-\bar\theta)\bar\theta^kw)$.
    Since $\sum_{k=0}^\infty (1-\bar\theta)\bar\theta^k<\infty$, the logarithmic derivatives
    of the partial products converge locally uniformly away from the
    zeros of $e_\theta$. Consequently,
    \begin{align*}
        \frac{e_\theta'(w)}{e_\theta(w)}=(\log e_\theta)'(w)=\sum_{k=0}^\infty\frac{(1-\bar\theta)\bar\theta^k}{1+(1-\bar\theta)\bar\theta^kw}.
    \end{align*}
    Thus, for $z>0$,
    \begin{align*}
        \frac{e_\theta'(iz)}{e_\theta(iz)}
        +\frac{e_\theta'(-iz)}{e_\theta(-iz)}
        &=
        2\sum_{k=0}^\infty
        \frac{(1-\bar\theta)\bar\theta^k}{1+((1-\bar\theta)\bar\theta^k)^2z^2}
        >0.
    \end{align*}
In particular, the quantity is real. For $0<\theta<1$, we can use the equality $e_\theta(w)=\frac{1}{e_{\bar\theta}(-w)}$ to derive the result.
\end{proof}
The next three corollaries follow directly from the representations
\begin{align}
     e_\theta(iz)&=\cos_\theta(z)+i\sin_\theta(z),\notag
        \\ \sin_\theta(z)&=\frac{e_\theta(iz)-e_\theta(-iz)}{2i},\qquad\text{ and}\qquad \cos_\theta(z)=\frac{e_\theta(iz)+e_\theta(-iz)}{2},\label{sin_sum_e}
     \end{align}
     together with Lemma~\ref{e_theta_singularities_thetale1} and Lemma~\ref{e_theta_entire_thetage1}.
\begin{Corollary}\label{sin_and_cos_prelim_1}
    For $0<\theta<1$, the functions $\sin_\theta$ and $\cos_\theta$ are meromorphic and have simple poles precisely at $z=\pm\frac{i}{(1-\theta)\theta^n}$, $n\in\N_0$.  
 \end{Corollary}
\begin{Corollary}\label{sin_and_cos_prelim_2}
    For $\theta\ge1$, the functions $\sin_\theta$ and $\cos_\theta$ are entire. 
\end{Corollary}
\begin{Corollary}\label{sin_and_cos_prelim_3}
   For $\theta>0$, $\sin_\theta$ and $\cos_\theta$ do not have any common zeros.
\end{Corollary}
\begin{Corollary}\label{arctan_cor}
   For $\theta >0$ and all $t\ge 0$, it holds $\arctan_{\theta}(t)=\arctan_{\bar\theta }(t)$. In particular, $\arctan_{\bar\theta }$ is well-defined.
\end{Corollary}
\begin{proof}
 The case $\theta=1$ is trivial, so without loss of generality, let $0<\theta<1$. The identity of meromorphic functions $ e_{\theta }(z)e_{\bar\theta }(-z)=1 $ implies $e_{\theta } (iz) = \frac{1}{e_{\bar\theta }(-iz)} $, so for all $z\in[0,z_c(\theta))$, it holds
    \begin{align*}
       \cos_\theta(z)&=\frac{e_\theta(iz)+e_\theta(-iz)}{2}
       =\frac{e_{\bar\theta}(iz)+e_{\bar\theta}(-iz)}{2e_{\bar\theta}(iz)e_{\bar\theta}(-iz)}=\frac{\cos_{\bar\theta}(z)}{e_{\bar\theta}(iz)e_{\bar\theta}(-iz)}.
    \end{align*}
    Lemma~\ref{e_theta_entire_thetage1} yields that $\cos_\theta$ and $\cos_{\bar\theta}$ have the same positive real zeros, which justifies the shorter notation $z_c\coloneqq z_c(\theta)=z_c(\bar\theta)$. A similar argument for $\sin_\theta$ and $\sin_{\bar\theta}$ implies that $\tan_{\bar\theta }$ is defined on the same interval $[0,z_c)$. Moreover, for $z\in[0,z_c(\theta))$, it holds
\begin{align*} 
    \tan_{\theta}(z)
    &= \frac{ e_{\theta } (iz)-e_{\theta }(-iz) }{ i\bigl(e_{\theta } (iz)+e_{\theta }(-iz)\bigr) } 
    = \frac{ e_{\bar\theta } (iz)-e_{\bar\theta }(-iz) }{ i\bigl( e_{\bar\theta } (iz)+e_{\bar\theta }(-iz) \bigr) } 
    = \tan_{\bar\theta }(z).
\end{align*} 
Recalling that $\tan_{\theta}:[0,z_c(\theta))\to[0,\infty)$ is a bijective function by Lemma 12 in \cite{AR} yields the claim.
\end{proof}
 \begin{Corollary}\label{sine_and_cose_prelim}
      For $\theta\ge1$ and $y<0$, the functions $\sin_\theta(z,y)$ and $\cos_\theta(z,y)$ are entire in $z$.
 \end{Corollary}
 \begin{proof}
    For $\theta>1$ and $y<0$, we can use the estimate
\begin{align*}
    |\sin_{\theta}(z, y)|\leq \sum_{k~\text{odd}}\frac{|z|^{k}}{[k]_{\theta}!}|e^{\theta\frac{1-\theta^{k}}{1-\theta} y}|\leq \sum_{k=0}^\infty\frac{|z|^{k}}{[k]_{\theta}!}=e_\theta(|z|)
\end{align*}
and a similar one for $\cos_{\theta}(z,y)$. Lemma~\ref{e_theta_entire_thetage1} then yields that the functions $\sin_\theta(z,y)$ and $\cos_\theta(z,y)$ are entire. For $\theta=1$, we use the continuous extensions $\sin_{\theta}(z, y)=\sin(ze^y)$ and $\cos_{\theta}(z, y)=\cos(ze^y)$.
 \end{proof}
For the remainder of this paper, we will use the following notation: For $k\in\N_0$, set
   \begin{align}\label{def_zk}
       z_k^+\coloneqq \frac{i}{(1-\theta)\theta^k},\qquad z_k^-\coloneqq \frac{-i}{(1-\theta)\theta^k}\qquad\text{and}\qquad Z_\theta\coloneqq \{z_k^+,z_k^-:k\in\N_0\}.
   \end{align}
   We next relate the shifted $q$-trigonometric functions to their unshifted counterparts in order to control their meromorphic continuations and residues.
 \begin{Lemma}\label{sin_e_and_cos_e_prelim}
    For $0<\theta<1$ and $y\in\R$, we have
     \begin{align}
         \sin_\theta(z,y)&=e^{\frac{\theta y}{1-\theta}}\sum_{m=0}^\infty\frac{\left(-\frac{\theta y}{1-\theta}\right)^m}{m!}\sin_\theta(\theta^m z)\quad\text{and}\label{meromorhpic_sin}
         \\\cos_\theta(z,y)&=e^{\frac{\theta y}{1-\theta}}\sum_{m=0}^\infty\frac{\left(-\frac{\theta y}{1-\theta}\right)^m}{m!}\cos_\theta(\theta^m z),\label{meromorhpic_cos}
     \end{align}
     where the right-hand sides can be seen as a meromorphic continuation to the whole complex plane, respectively.
 \end{Lemma}
 \begin{proof}
Let $0<\theta<1$ and $|z|<\frac{1}{1-\theta}$. 
Then, using the exponential series,
\begin{align}\label{eq_for_proof_sin_zy}
    \sin_\theta(z,y)
    &= \sum_{j=0}^\infty\frac{(-1)^{j}z^{2j+1}}{[2j+1]_{\theta}!}e^{\theta\frac{1-\theta^{2j+1}}{1-\theta} y}\notag
    \\&=e^{\frac{\theta y}{1-\theta}}\sum_{j=0}^\infty\frac{(-1)^{j}z^{2j+1}}{[2j+1]_{\theta}!}\sum_{m=0}^\infty\frac{\left(-\frac{ \theta y}{1-\theta}\right)^m}{m!}\theta^{(2j+1)m}\notag
    \\&=e^{\frac{\theta y}{1-\theta}}\sum_{m=0}^\infty\frac{\left(-\frac{\theta y}{1-\theta}\right)^m}{m!}\sum_{j=0}^\infty\frac{(-1)^{j}(\theta^m z)^{2j+1}}{[2j+1]_{\theta}!}\notag
    \\&=e^{\frac{\theta y}{1-\theta}}\sum_{m=0}^\infty\frac{\left(-\frac{\theta y}{1-\theta}\right)^m}{m!}\sin_\theta(\theta^m z),
\end{align}
where we exchanged the sums since, apart from the
finite factor $e^{\frac{\theta y}{1-\theta}}$,
\begin{align*}
    \sum_{j=0}^\infty\sum_{m=0}^\infty
    \frac{|z|^{2j+1}}{[2j+1]_\theta!}
    \frac{\left(\frac{\theta|y|}{1-\theta}\right)^m}{m!}
    \theta^{(2j+1)m}
    &\leq \left(\sum_{k=0}^\infty \frac{|z|^{k}}{[k]_\theta!}\right)\exp\left(\frac{\theta|y|}{1-\theta}\right)
    <\infty
\end{align*}
for $|z|<1/(1-\theta)$, by Lemma~\ref{e_theta_singularities_thetale1}.
By Corollary~\ref{sin_and_cos_prelim_1}, the function $\sin_\theta$ is meromorphic and has simple poles at  $z=\pm\frac{i}{(1-\theta)\theta^n}$, $n\in\N_0$. Thus, for fixed $m$, the possible poles of $\sin_\theta(\theta^m z)$ are given by $z=\pm\frac{i}{(1-\theta)\theta^{n+m}}\in Z_\theta$. Let $K$ be a compact subset of $\C\backslash Z_\theta$. For every $z\in K$, the choice $0<\theta<1$ yields that $\sup_{z\in K}|\theta^mz|\to 0$ as $m\to\infty$, so in particular, there exists some $m_0\in\N_0$ such that
$|\theta^m z|\le\frac{1}{2(1-\theta)}$ for all $z\in K$ and $m\ge m_0$. Again, by Corollary~\ref{sin_and_cos_prelim_1}, $\sin_\theta$ is holomorphic on the compact disk $|z|\le\frac{1}{2(1-\theta)}$ and thus is bounded there. In particular, with $C_K\coloneqq \max_{|w|\le\frac{1}{2(1-\theta)}}|\sin_\theta(w)|>0$, we have
   $|\sin_\theta(\theta^m z)|\le C_K$,
for all $z\in K$ and $m\ge m_0$. Therefore, for $z\in K$, the right-hand side in \eqref{eq_for_proof_sin_zy} can be estimated by
\begin{align*}
    \left|e^{\frac{\theta y}{1-\theta}}\sum_{m=0}^\infty\frac{\left(-\frac{\theta y}{1-\theta}\right)^m}{m!}\sin_\theta(\theta^m z)\right|
    \!\le\! e^{\frac{\theta y}{1-\theta}}\!\left(\sum_{m=0}^{m_0-1}\left|\frac{\left(-\frac{\theta y}{1-\theta}\right)^m}{m!}\sin_\theta(\theta^m z)\right|\!+\!\!\!\sum_{m=m_0}^\infty\!\!\frac{\left(\frac{\theta |y|}{1-\theta}\right)^m}{m!}C_K\right)\!<\!\infty.
\end{align*}
The first sum of the finitely many continuous terms with $m\le m_0-1$ is bounded on the compact set $K$ and the remaining terms are uniformly dominated on $K$ as well, so the series in \eqref{eq_for_proof_sin_zy} converges locally uniformly on $\C\backslash Z_\theta$. Moreover, it defines a holomorphic function on this set and since it agrees with $ \sin_\theta(z,y)$ on the disk $|z|<\frac{1}{1-\theta}$, the series provides a meromorphic continuation of $ \sin_\theta(z,y)$.
The proof for $\cos_\theta(z,y)$ follows by a similar argument.
 \end{proof}
 \begin{Lemma}\label{residue_sin_cos_prelim}
   For $0<\theta<1$ and $y\in\R$, every singularity $z_0$ of 
   $\sin_\theta(z,y)$ or $\cos_\theta(z,y)$ is at most a simple pole and it holds $z_0\in Z_\theta$. Furthermore, for every $k\in\N_0$, it holds
   \begin{align}
   \mathrm{Res}_{z=z_k^+}\sin_\theta(z)&= i\mathrm{Res}_{z=z_k^+}\cos_\theta(z),\label{Res_sin_cos_pos}
   \\\mathrm{Res}_{z=z_k^-}\sin_\theta(z)&= -i\mathrm{Res}_{z=z_k^-}\cos_\theta(z),\label{Res_sin_cos_neg}
      \\ \mathrm{Res}_{z=z_k^+}\sin_\theta(z,y)&= i\mathrm{Res}_{z=z_k^+}\cos_\theta(z,y),\quad\text{and}\notag
       \\ \mathrm{Res}_{z=z_k^-}\sin_\theta(z,y)&=-i \mathrm{Res}_{z=z_k^-}\cos_\theta(z,y).\notag
   \end{align}
 \end{Lemma}
 \begin{proof}
    Fix $k\in\N_0$. By Lemma~\ref{sin_e_and_cos_e_prelim}, we can use the meromorphic continuations \eqref{meromorhpic_sin} and \eqref{meromorhpic_cos}
     and therefore start with the properties of the singularities of $\sin_\theta(\theta^m z)$ and $\cos_\theta(\theta^m z)$ and then transfer them to $\sin_\theta(z,y)$ and $\cos_\theta(z,y)$. The functions $\sin_\theta(\theta^m z)$ and $\cos_\theta(\theta^m z)$, respectively, can only have a singularity at $z=z_k^\pm$ if $m\le k$ since 
    \begin{align*}
        \theta^mz_k^\pm=\frac{\pm i}{(1-\theta)\theta^{k-m}}
    \end{align*}
    is equal to $ z_{k-m}^\pm\in Z_\theta$ if and only if $m\le k$. Since $k$ is fixed, there are only finitely many summands on the right-hand side in \eqref{meromorhpic_sin} and \eqref{meromorhpic_cos}, respectively, with a singularity in $z_k^\pm$. Moreover, all poles of $\sin_\theta$ and $\cos_\theta$ are simple, so the singularities of $\sin_\theta(z,y)$ and $\cos_\theta(z,y)$ are at most simple poles, which proves the first part of the claim.
    
    The residue relations will follow from the representations \eqref{sin_sum_e}. 
    First, consider $z_k^+$, so $-iz_k^+=\frac{1}{(1-\theta)\theta ^k}$ and $iz_k^+=\frac{-1}{(1-\theta)\theta ^k}$. Hence, by Lemma~\ref{e_theta_singularities_thetale1}, the function $z\mapsto e_\theta(-iz)$ has a simple pole at $z=z_k^+$ and $z\mapsto e_\theta(iz)$ is holomorphic at $z=z_k^+$. Therefore, by \eqref{sin_sum_e},
    it holds
    \begin{align*}
        \mathrm{Res}_{z=z_k^+}\sin_\theta(z)=\frac{-\mathrm{Res}_{z=z_k^+}e_\theta(-iz)}{2i}=
        \frac{i\mathrm{Res}_{z=z_k^+}e_\theta(-iz)}{2}
    \end{align*}
    and
    \begin{align*}
        \mathrm{Res}_{z=z_k^+}\cos_\theta(z)=\frac{\mathrm{Res}_{z=z_k^+}e_\theta(-iz)}{2},
    \end{align*}
    which yields \eqref{Res_sin_cos_pos}.
    The same idea applied to $z_k^-$ gives \eqref{Res_sin_cos_neg}.
    
    Now, we transfer these identities to the scaled functions $\sin_\theta(\theta^m z)$ and $\cos_\theta(\theta^m z)$. As we have seen above, for $m\le k$, it holds with $\theta^m z_k^\pm=z_{k-m}^\pm$
    \begin{align*}
        \mathrm{Res}_{z=z_k^\pm}\sin_\theta(\theta ^mz)=\theta^{-m}\mathrm{Res}_{z=z_{k-m}^\pm}\sin_\theta(z)
    \end{align*}
    and analogously for $\cos_\theta(\theta^m z)$. For $m>k$, there is no singularity. Therefore, the residue relations between $\sin_\theta(z)$ and $\cos_\theta(z)$ also hold true between $\sin_\theta(\theta^mz)$ and $\cos_\theta(\theta^m z)$.
    
    The last step is to deduce the claim for $\sin_\theta(z,y)$ and $\cos_\theta(z,y)$. Using  the meromorphic continuation given in \eqref{meromorhpic_sin} and \eqref{meromorhpic_cos} and the fact that only finitely many terms of the sum on their right-hand side contribute to the residue at $z_k^\pm$, we may take residues term by term. Therefore, it holds for the residue at $z_k^\pm$
    \begin{align*}
        \mathrm{Res}_{z=z_k^+}\sin_\theta(z,y)&=e^{\frac{\theta y}{1-\theta}}\sum_{m=0}^k\frac{\left(-\frac{\theta y}{1-\theta}\right)^m}{m!}\mathrm{Res}_{z=z_k^+}\sin_\theta(\theta^m z)
        \\&=ie^{\frac{\theta y}{1-\theta}}\sum_{m=0}^k\frac{\left(-\frac{\theta y}{1-\theta}\right)^m}{m!}\mathrm{Res}_{z=z_k^+}\cos_\theta(\theta^m z)
        \\&=i\mathrm{Res}_{z=z_k^+}\cos_\theta(z,y)
    \end{align*}
    and the same for $\cos_\theta(z,y)$ and at $z_k^-$, respectively. This shows the claim.
 \end{proof}
\section{The case \texorpdfstring{$\theta>0$}{theta>0}}\label{sec_theta_pos}
In the first subsection, we find a closed formula for the generating function for all $\theta>0,\theta\neq 1$.
Then, we consider the cases $0<\theta<1$ and $\theta>1$ separately in Section~\ref{sec_thetale1} and Section~\ref{sec_thetage1}, respectively. Finally, the invariant distribution of both cases is treated in Section~\ref{invariant distributions}.
\subsection{Generating functions}
Lemma 4 for $y\geq0$ and Lemma 5 for $y\leq 0$ from \cite{AR} give an explicit formula for the persistence probability $p_n^{\xi,y}(\theta)$:
\begin{Lemma}\label{P_x_theta_positiv}
    Let $\theta>0$ and $\xi\in[0,1]$. For all $n\geq1$, it holds
        \begin{align*}
        p_n^{\xi,y}(\theta)=\begin{cases}
            c_ne^{-\alpha_ny}, & y\geq 0,
        \\
        \sum_{j=1}^n \frac{(-\xi)^{j-1} r_{n-j+1}}{[j]_\theta!} \left(1 - e^{\alpha_j y} \right)+c_n, &y\leq 0,
        \end{cases}
    \end{align*}
where $\alpha_1\coloneqq\theta$, $c_1\coloneqq 1-\xi$, $r_1\coloneqq \xi$ and the further constants for $n\ge 2$ are defined via the recursions
\begin{align*}
    \alpha_n&\coloneqq (\alpha_{n-1}+1)\theta=\sum_{i=1}^n\theta^i,
    \\c_n&\coloneqq\frac{c_{n-1}(1-\xi)}{\alpha_{n-1}+1}=(1-\xi)^n\prod_{j=1}^{n-1}\frac{1}{\sum_{i=0}^j\theta^i}=\frac{(1-\xi)^n}{[n]_\theta!},\qquad \text{and}
    \\r_n&\coloneqq\sum_{j=1}^{n-1}\frac{(-1)^{j-1}\xi^jr_{n-j}}{\prod_{m=1}^{j-1} (\alpha_m + 1)}+\xi c_{n-1}=-\sum_{j=1}^{n-1}\frac{(-\xi)^jr_{n-j}}{[j]_\theta!}+\xi\frac{(1-\xi)^{n-1}}{[n-1]_\theta!}.
\end{align*}
\end{Lemma}
\begin{Remark}
    Note that for all $j\ge 1$, $[j]_\theta!=\prod_{m=1}^{j-1} (\alpha_m + 1)$. For $\theta\neq1$, we have $\alpha_n=\frac{\theta(1-\theta^n)}{1-\theta}$.
\end{Remark}
The recursion for $r_n$ can be encoded in the following generating-function identity.
\begin{Lemma}\label{def_R}
    For $\theta>0$, $\xi\in[0,1]$ and $R(z)\coloneqq \sum_{n=1}^{\infty}r_nz^n$, we have, on the common disk of convergence,
    \begin{align*}
        R(z)=\xi z\frac{e_\theta((1-\xi)z)}{e_\theta(-\xi z)}.
    \end{align*}
\end{Lemma}
\begin{proof}
        See \cite{AR}, Lemma 6 for the calculation. For $\xi=0$, we have $R\equiv 0$ and $R$ has an infinite radius of convergence. For $\theta=1$, $R(z)=\xi ze^z$ is entire. For $0<\xi<1$ and $\theta<1$, by Lemma~\ref{e_theta_singularities_thetale1}, the radius of convergence is given by the singularity with smallest modulus of $e_\theta((1-\xi)z)$. For $\xi=1$ and $\theta<1$, $e_\theta(0)=1$ has no singularities, so the radius of convergence is infinite. For $0<\xi\le1$ and $\theta> 1$, the function $e_\theta$ is entire by Lemma~\ref{e_theta_entire_thetage1}, so the radius of convergence is given by the zero with smallest modulus of $e_\theta(-\xi z)$. 
    \end{proof}
    For the remainder of Section~\ref{sec_theta_pos}, consider $\theta\neq 1$.
Combining Lemmas~\ref{P_x_theta_positiv} and \ref{def_R} yields the generating function for an arbitrary starting point.
\begin{Lemma}\label{gf_pos_theta}
   For $\theta>0$, $\theta\neq 1$ and $\xi\in[0,1]$, we have
    \begin{align*}
      \hat{P}_{\xi,y,\theta}(z)=\begin{cases}
        \sum_{j=0}^\infty \frac{((1-\xi)z)^j}{[j]_\theta!}e^{-\frac{\theta(1-\theta^j)}{1-\theta}y},&y>0,
          \\\frac{e_\theta((1-\xi)z)}{e_\theta(-\xi z)}\sum_{j=0}^\infty \frac{(-\xi z)^{j}}{[j]_\theta!}e^{\frac{\theta(1-\theta^j)}{1-\theta}y},~~&y\le 0
      \end{cases}
    \end{align*}
    on the common disk of convergence and at most within the radius of convergence of $R$ defined in Lemma~\ref{def_R}.
\end{Lemma}
\begin{proof}
    For $y\ge 0$, Lemma~\ref{P_x_theta_positiv} yields
    \begin{align*}
        \hat{P}_{\xi,y,\theta}(z)
        &=1+\sum_{j=1}^\infty c_je^{-\alpha_jy}z^j
        =\sum_{j=0}^\infty \frac{((1-\xi)z)^j}{[j]_\theta!}e^{-\frac{\theta(1-\theta^j)}{1-\theta}y}.
    \end{align*}
    For $y\le 0$, Lemma~\ref{P_x_theta_positiv} and Lemma~\ref{def_R} yield
    \begin{align*}
        \hat{P}_{\xi,y,\theta}(z)-1
        &=\sum_{n=1}^\infty \sum_{j=1}^n \frac{(-\xi)^{j-1} r_{n-j+1}}{[j]_\theta!} \left(1 - e^{\alpha_j y} \right)z^n+\sum_{n=1}^\infty c_nz^n
         \\&=\sum_{j=1}^\infty \frac{(-\xi)^{j-1}}{[j]_\theta!} \left(1 - e^{\alpha_j y} \right)z^{j-1}\sum_{n=j}^\infty r_{n-j+1}z^{n-j+1}+\sum_{n=1}^\infty \frac{((1-\xi)z)^n}{[n]_\theta!}
        \\&=\sum_{j=1}^\infty \frac{(-\xi z)^{j-1} }{[j]_\theta!} \left(1 - e^{\alpha_j y} \right)R(z)+e_\theta((1-\xi)z)-1
        \\&=R(z)\left(\sum_{j=1}^\infty \frac{(-\xi z)^{j-1} }{[j]_\theta!}-\sum_{j=1}^\infty \frac{(-\xi z)^{j-1}}{[j]_\theta!}e^{\alpha_j y}\right)+e_\theta((1-\xi)z)-1
        \\&= \frac{e_\theta((1-\xi)z)}{e_\theta(-\xi z)}\left(-(e_\theta(-\xi z)-1)-\xi z\sum_{j=1}^\infty \frac{(-\xi z)^{j-1}}{[j]_\theta!}e^{\alpha_j y}\right)+e_\theta((1-\xi)z)-1
        \\&=e_\theta((1-\xi)z)\left(-1+\frac{1}{e_\theta(-\xi z)}+ \frac{1}{e_\theta(-\xi z)}\sum_{j=1}^\infty \frac{(-\xi z)^{j}}{[j]_\theta!}e^{\alpha_j y}+1\right)-1
        \\&=\frac{e_\theta((1-\xi)z)}{e_\theta(-\xi z)}\sum_{j=0}^\infty \frac{(-\xi z)^{j}}{[j]_\theta!}e^{\frac{\theta(1-\theta^j)}{1-\theta}y}-1.
    \end{align*}
   For $z$ in the common disk of absolute convergence, we have
\begin{align*}
    \sum_{n=1}^\infty\sum_{j=1}^n
    \left|\frac{(-\xi)^{j-1}r_{n-j+1}}{[j]_\theta!} \bigl(1-e^{\alpha_jy}\bigr)z^n
    \right|
    \leq\left(\sum_{j=1}^\infty\frac{\xi^{j-1}|z|^{j-1}}{[j]_\theta!}\right)
    \left( \sum_{k=1}^\infty |r_k||z|^k\right)
    <\infty.
\end{align*}
Thus, the order of summation in the calculations above may be exchanged. 
\end{proof}
\subsection{The case \texorpdfstring{$0<\theta<1$}{0<theta<1}}\label{sec_thetale1}
In this subsection, assume $0<\theta<1$, $\xi<1$ and set $\lambda\coloneqq (1-\xi)(1-\theta)$. The case $0<\theta<1$ and $\xi=1$ is treated in Section~\ref{trivialcases}. First, we consider $y>0$ and the rate of decay of the $(p_n^{\xi,y}(\theta))_n$ from Lemma~\ref{P_x_theta_positiv}.
\begin{Lemma}\label{asymptotic_pos_theta_pos_y}
  For $0<\theta<1$, $\xi\in[0,1)$ and $y>0$, as  $n\to\infty$,
    \begin{align*}
        p_n^{\xi,y}(\theta)\sim \frac{1}{\lambda(\theta;\theta)_\infty}e^{\frac{\theta}{\theta-1}y}\lambda^{n+1}.
    \end{align*}
 \end{Lemma}
\begin{proof}
 By Lemma~\ref{P_x_theta_positiv},
    \begin{align*}
        p_n^{\xi,y}(\theta)
        &=c_ne^{-\alpha_ny}
        =\frac{(1-\xi)^n}{[n]_\theta!}e^{-\alpha_ny}=\frac{(1-\xi)^n(1-\theta)^n}{[n]_\theta!(1-\theta)^n}e^{-\frac{\theta(1-\theta^n)}{1-\theta}y}
        =\frac{\lambda^n}{(\theta;\theta)_n}e^{\frac{\theta(1-\theta^n)}{\theta-1}y}
        \\&\sim\frac{1}{\lambda(\theta;\theta)_\infty}e^{\frac{\theta}{\theta-1}y}\lambda^{n+1}.\qedhere
    \end{align*}
\end{proof}
Now, we treat the case $y\le0$ by deducing the asymptotic behavior of $(p_n^{\xi,y}(\theta))$ from its generating function.
To abbreviate, set 
$\ell(z)\coloneqq e_\theta((1-\xi)z)$ and 
$$v(z)\coloneqq v_y(z)\coloneqq  \frac{1}{e_\theta(-\xi z)}\sum_{j=0}^\infty \frac{(-\xi z)^{j}}{[j]_\theta!}e^{\frac{\theta(1-\theta^j)}{1-\theta}y}.$$
Note that the following result also holds for $\xi=1$.
\begin{Lemma}\label{singularities_pos_theta}
   For $0<\theta<1$, $\xi\in\big[0,1]$ and $y\le 0$, we have on the common disk of convergence
   \begin{align}\label{repr_v_thetale1}
    v(z)=e^{\frac{\theta y}{1-\theta}}\sum_{m=0}^\infty\frac{\left(-\frac{\theta y}{1-\theta}\right)^m}{m!}(-\xi z(1-\theta);\theta)_m.
\end{align}
The right-hand side defines an entire function of $z$ and hence gives an entire analytic continuation of $v$.
\end{Lemma}
\begin{proof}
If $\xi=0$, then $v(z)\equiv1$, so the claim is immediate.
In the remainder of the proof, assume $\xi>0$. Fix $|z|<\frac{1}{\xi(1-\theta)}$.  
Then, using the exponential series,
\begin{align*}
    \sum_{j=0}^\infty \frac{(-\xi z)^{j}}{[j]_\theta!}e^{\frac{\theta(1-\theta^j)}{1-\theta}y}
    &=e^{\frac{\theta y}{1-\theta}}\sum_{m=0}^\infty\frac{\left(-\frac{\theta y}{1-\theta}\right)^m}{m!}\sum_{j=0}^\infty \frac{(-\xi z\theta^m)^{j}}{[j]_\theta!}
    =e^{\frac{\theta y}{1-\theta}}\sum_{m=0}^\infty\frac{\left(-\frac{\theta y}{1-\theta}\right)^m}{m!}e_\theta(-\xi z\theta^m),
\end{align*}
where we exchanged the sums since 
\begin{align*}
    \sum_{m=0}^\infty\sum_{j=0}^\infty
    \frac{\left(-\frac{\theta y}{1-\theta}\right)^m}{m!}
    \frac{(\xi|z|\theta^m)^j}{[j]_\theta!}
    \leq\exp\left(-\frac{\theta y}{1-\theta}\right)e_\theta(\xi|z|)
    <\infty
\end{align*}
for $y\le 0$ and $|z|<1/(\xi(1-\theta))$.
In conclusion, 
\begin{align*}
    v(z)=e^{\frac{\theta y}{1-\theta}}\sum_{m=0}^\infty\frac{\left(-\frac{\theta y}{1-\theta}\right)^m}{m!}\frac{e_\theta(-\xi z\theta^m)}{e_\theta(-\xi z)}.
\end{align*}
Now we use the product representation of $e_\theta(\cdot)$ from \eqref{e_theta_prod_thetale1} and get
\begin{align*}
    \frac{e_\theta(-\xi z\theta^m)}{e_\theta(-\xi z)}
    =\frac{(-\xi z(1-\theta);\theta)_\infty}{(-\xi z\theta^m(1-\theta);\theta)_\infty}
    =(-\xi z(1-\theta);\theta)_m,
\end{align*}
which implies (\ref{repr_v_thetale1}).
By definition,
\begin{align*}
    |(-\xi z(1-\theta);\theta)_m|
    \le\prod_{k=0}^{m-1}(1+\xi|z|(1-\theta)\theta^k)
    \le\prod_{k=0}^{\infty}(1+\xi|z|(1-\theta)\theta^k)
   \! =\!(-\xi |z|(1-\theta);\theta)_\infty,
\end{align*}
where we used $1+\xi|z|(1-\theta)\theta^k\ge 1$ in the second step. Note that the upper bound is a constant in $m$. Moreover, since $\sum_{k=0}^\infty\xi|z|(1-\theta)\theta^k=\xi|z|(1-\theta)\sum_{k=0}^\infty\theta^k<\infty$ for $0<\theta<1$, it holds that $0\le (-\xi |z|(1-\theta);\theta)_\infty<\infty$. For every $R>0$, the preceding estimate holds uniformly for $|z|\leq R$, with the finite bound
$C_R\coloneqq(-\xi R(1-\theta);\theta)_\infty$. Hence, by the Weierstrass $M$-test, the right-hand side of the representation in \eqref{repr_v_thetale1} converges locally uniformly on $\C$ and therefore defines an entire function.
\end{proof}
The next lemma identifies the candidate eigenfunction, verifies its
eigenvalue equation, and proves its strict positivity. Since the
positivity argument relies on the eigenvalue equation, we establish
the latter directly rather than applying
Lemma~\ref{eigenfunction_lemma}.
\begin{Lemma}\label{h_eigenfct_pos_theta}
    For $0<\theta<1$ and $\xi\in[0,1)$, the function $h:\R\to\R$
    \begin{align}\label{def_h_pos_theta_2}
     h(y)\coloneqq h_{\theta,\xi}(y)\coloneqq  \begin{cases}
          e^{\frac{\theta}{\theta-1}y},~~~&y>0,
          \\e^{\frac{\theta y}{1-\theta}}\sum_{m=0}^\infty\frac{\left(-\frac{\theta y}{1-\theta}\right)^m}{m!}\left(\frac{-\xi}{\lambda}(1-\theta);\theta\right)_m,~~~~&y\le 0,
      \end{cases}
  \end{align}
  satisfies the eigenvalue equation \eqref{eq_eigenfunction} with eigenvalue $\lambda$. Further,  $h(x)>0$ for all $x\in\R$.
\end{Lemma}
\begin{proof}
For $x>0$, we have $\theta x>0$, so that the computation for positive arguments is straightforward.
For $x\le0$, we have $\theta x\le0$, so \eqref{eq_eigenfunction} is equivalent to
  \begin{align}\label{eq_eigenvalue_concrete}
      \lambda h(x)=\int_{(\theta x,0)}\!\!\!\!\!\!\!\xi  e^yh(y)\d y+\int_{(0,\infty)}\!\!\!\!\!\!\!(1-\xi)  e^{\frac{1}{\theta-1}y}\d y=\int_{(\theta x,0)}\!\!\!\!\!\!\!\xi  e^yh(y)\d y+\lambda. 
  \end{align}
 For every $m\ge 0$, the substitution $t=\frac{y}{\theta-1}$ gives
  \begin{align*}
     \int_{(\theta x,0)} e^ye^{\frac{\theta y}{1-\theta}}\left(-\frac{\theta y}{1-\theta}\right)^m\d y
     &=\theta^m(\theta-1)\int_{-\frac{\theta x}{1-\theta}}^0 t^m e^{-t}\d t
     \\&=(\theta-1)\theta^mm!\left(e^{\frac{\theta x}{1-\theta}}\sum_{k=0}^m\frac{1}{k!}\left(\frac{-\theta x}{1-\theta}\right)^k-1\right).
  \end{align*}
  Set $Q_m\coloneqq \left(\frac{-\xi}{\lambda}(1-\theta);\theta\right)_m$. We have, by definition of the $q$-Pochhammer symbol, 
  \begin{align*}
      Q_{m+1}=Q_m\left(1+\frac{\xi}{\lambda}(1-\theta)\theta^m\right),
  \end{align*}
  so that
      $\lambda\left(Q_{m}-Q_{m+1}\right)=\xi(\theta-1)\theta^m Q_m$.
  For $h$ given in \eqref{def_h_pos_theta_2}, we thus have
  \begin{align}\label{calculation_eigenfunction}
      &\int_{(\theta x,0)}\xi  e^yh(y)\d y+\lambda \notag
      \\&=\sum_{m=0}^\infty\xi\frac{Q_m}{m!}\int_{(\theta x,0)} e^ye^{\frac{\theta y}{1-\theta}}\left(-\frac{\theta y}{1-\theta}\right)^m\d y+\lambda\notag
    \\&=\sum_{m=0}^\infty\frac{\xi(\theta-1)\theta^mQ_m}{m!}m!\left(e^{\frac{\theta x}{1-\theta}}\sum_{k=0}^m\frac{1}{k!}\left(\frac{-\theta x}{1-\theta}\right)^k-1\right)+\lambda\notag
    \\&=\lambda\sum_{m=0}^\infty \left(Q_{m}-Q_{m+1}\right)\left(e^{\frac{\theta x}{1-\theta}}\sum_{k=0}^m\frac{1}{k!}\left(\frac{-\theta x}{1-\theta}\right)^k-1\right)+\lambda,
  \end{align}
where one could exchange the sum and the integral since the sum is absolutely convergent by Lemma~\ref{singularities_pos_theta}. For every fixed $N\in\N$, it holds, using summation by parts,
\begin{align*}
    \sum_{m=0}^N \left(Q_{m}-Q_{m+1}\right)\sum_{k=0}^m\frac{1}{k!}\left(\frac{-\theta x}{1-\theta}\right)^k
    =\sum_{m=0}^NQ_m\frac{1}{m!}\left(\frac{-\theta x}{1-\theta}\right)^m-Q_{N+1}\sum_{k=0}^{N}\frac{1}{k!}\left(\frac{-\theta x}{1-\theta}\right)^k,
\end{align*}
and
\begin{align*}
    &\lambda\sum_{m=0}^N \left(Q_{m}-Q_{m+1}\right)\left(e^{\frac{\theta x}{1-\theta}}\sum_{k=0}^m\frac{1}{k!}\left(\frac{-\theta x}{1-\theta}\right)^k-1\right)+\lambda
   \\&=\lambda e^{\frac{\theta x}{1-\theta}}\sum_{m=0}^NQ_m\frac{1}{m!}\left(\frac{-\theta x}{1-\theta}\right)^m+\lambda Q_{N+1}\left(1-e^{\frac{\theta x}{1-\theta}}\sum_{k=0}^{N}\frac{1}{k!}\left(\frac{-\theta x}{1-\theta}\right)^k\right),
\end{align*}
where the second term converges to zero as $N\to\infty$: On the one hand, 
\begin{align*}
    |Q_{N+1}|\le \left(\frac{-\xi}{\lambda}(1-\theta);\theta\right)_\infty
\end{align*}
and on the other hand, $\lim_{N\to\infty} 1-e^{\frac{\theta x}{1-\theta}}\sum_{k=0}^{N}\frac{1}{k!}\left(\frac{-\theta x}{1-\theta}\right)^k=0$. In particular, letting $N\to\infty$ and combining with \eqref{calculation_eigenfunction}, we get
\begin{align*}
    &\int_{(\theta x,0)}\xi  e^yh(y)\d y+\lambda
    =\lambda e^{\frac{\theta x}{1-\theta}}\sum_{m=0}^\infty Q_m\frac{1}{m!}\left(\frac{-\theta x}{1-\theta}\right)^m
    =\lambda h(x).
\end{align*}
Therefore, this choice of $h$ satisfies the eigenvalue equation \eqref{eq_eigenvalue_concrete} and thus, \eqref{eq_eigenfunction}. It remains to show that $h$ is strictly positive. By definition, $h$ is continuous with $h(x)>0$ for $x>0$ and $h(0)=1$. Therefore, there is some $r>0$ such that $h(x)>0$ for all $x\ge-r$. For any $x$ with $\theta x\ge -r$, the eigenvalue equation \eqref{eq_eigenfunction} now implies $\lambda h(x)>0$, so we get $h(x)>0$ for $x\ge\frac{-r}{\theta}$. Iterating this argument, we get $h(x)>0$ for $x\ge\frac{-r}{\theta^n}$ for all $n\in\N$ and the claim follows with $\theta^n\to0$ for $n\to\infty$.
\end{proof}
We can now identify the unique dominant singularity for $y\leq 0$ and subsequently extract the coefficient asymptotics.
\begin{Lemma}\label{gf_pos_theta_le1}
  For $0<\theta<1$, $\xi\in[0,1)$ and $y\le 0$, we have
    \begin{align*}
      \hat{P}_{\xi,y,\theta}(z)= \ell(z)v(z)
    \end{align*}
    with radius of convergence $\frac{1}{\lambda}$. More concretely, the unique singularity with smallest modulus of $\hat{P}_{\xi,y,\theta}(z)$ is simple, isolated and given by $\frac{1}{\lambda}$.
\end{Lemma}
\begin{proof}
    For $y\le 0$, Lemma~\ref{gf_pos_theta} implies $\hat{P}_{\xi,y,\theta}(z)=\ell(z)v(z)$ within the common radius of convergence, but at most within the radius of convergence of $R$. By Lemma~\ref{e_theta_singularities_thetale1}, the singularities of the function $\ell$ are isolated and simple, and the unique singularity with smallest modulus is given by $\frac{1}{\lambda}$. By Lemma~\ref{singularities_pos_theta}, $v$ has an analytic continuation which is strictly positive at $z=\frac{1}{\lambda}$ by Lemma~\ref{h_eigenfct_pos_theta}. By the same argument, the radius of convergence of $R(z)$ used in Lemma~\ref{def_R} is also $\frac{1}{\lambda}$ for $\xi>0$ and it is known to be infinite for $\xi=0$, so that $\hat{P}_{\xi,y,\theta}(z)=\ell(z)v(z)$ holds for all $|z|<\frac{1}{\lambda}$. In particular, $\frac{1}{\lambda}$ is the unique singularity with smallest modulus of $\hat{P}_{\xi,y,\theta}(z)$ and it is simple and isolated.
\end{proof}
\begin{Lemma}\label{asymptotic_pos_theta_neg_y}
  For $0<\theta<1$, $\xi\in[0,1)$ and $y\le 0$, we have, with $h$ defined in \eqref{def_h_pos_theta_2}, as $n\to\infty$,
  \begin{align*}
       p_n^{\xi,y}(\theta)\sim \frac{1}{\lambda(\theta;\theta)_\infty}\,h(y)\lambda^{n+1}.
  \end{align*}
\end{Lemma}
\begin{proof}
    We want to apply Lemma~\ref{asymptotic_powerseries}, so we check that there is a function $f$ with
    \begin{align}
        \hat{P}_{\xi,y,\theta}(z)=\frac{f(z)}{\frac{1}{\lambda}-z}~~~~\text{for all}~|z|<\frac{1}{\lambda}\label{eq_Lemma_1},
    \end{align}
   where $f$ is analytic on $|z|<z_1$ with $\frac{1}{\lambda}<z_1$ and $f\left(\frac{1}{\lambda}\right)\neq 0$. 
   By Lemma~\ref{gf_pos_theta_le1}, $f(z)=\left(\frac{1}{\lambda}-z\right)\ell(z)v(z)$  is a suitable choice for \eqref{eq_Lemma_1}. Lemma~\ref{gf_pos_theta_le1} yields that $\frac{1}{\lambda}$ is the unique singularity with smallest modulus of $\ell(z)v(z)$ and it is simple and isolated. In particular, there is some $z_1>\frac{1}{\lambda}$ such that the function $f$ is analytic on $|z|<z_1$ and, by \eqref{repr_v_thetale1} and Lemma~\ref{h_eigenfct_pos_theta}, $h(y)=v_y\left(\frac{1}{\lambda}\right)>0$. The remaining part now follows from the representation in \eqref{e_theta_prod_thetale1}
   \begin{align*}
       \left(\frac{1}{\lambda}-z\right)\ell(z)=\left(\frac{1}{\lambda}-z\right)\frac{1}{1-z\lambda}\prod_{n=1}^\infty\frac{1}{1-z\lambda \theta^n}=\frac{1}{\lambda}\prod_{n=1}^\infty\frac{1}{1-z\lambda \theta^n},
   \end{align*}
   which is equal to $\frac{1}{\lambda}\prod_{n=1}^\infty\frac{1}{1-\theta^n}=\frac{1}{\lambda(\theta;\theta)_\infty}>0$ at $z=\frac{1}{\lambda}$, so $f\left(\frac{1}{\lambda}\right)=\frac{1}{\lambda(\theta;\theta)_\infty}h(y)>0$.
  Thus, $f$ satisfies the  conditions for Lemma \ref{asymptotic_powerseries}, which now yields $p_n^{\xi,y}(\theta)\sim \frac{1}{\lambda(\theta;\theta)_\infty}h(y)\lambda^{n+1}$.
\end{proof}
\begin{proof}[Proof of Theorem~\ref{maintheorem} (i)]
    By Lem\-ma~\ref{h_eigenfct_pos_theta}, $h$ satisfies the eigenvalue equation \eqref{eq_eigenfunction}. To use Lem\-ma~\ref{Kernel_lim_process}, first note that the asymptotics from both Lemma~\ref{asymptotic_pos_theta_pos_y} and Lemma~\ref{asymptotic_pos_theta_neg_y} have the same constant $c(\theta,\xi)=\frac{1}{\lambda(\theta;\theta)_\infty}$. Moreover, Equation (5) and Theorem 1 (d) from \cite{AR}
   yield that there is some constant $c_2(\theta,\xi)$ such that for $n\to\infty$, $p_n^{\xi}(\theta)\sim c_2(\theta,\xi)\lambda^{n+1}$, so
      Lemma~\ref{Kernel_lim_process} yields the claim.
\end{proof}
\subsection{The case \texorpdfstring{$\theta>1$}{theta>1}}\label{sec_thetage1}
Now consider $\theta >1$, $\xi>0$. The special case $\xi=0$ is treated in Section~\ref{trivialcases}. Recall the notation $\bar\theta\coloneqq\theta^{-1}\in(0,1)$ and fix $\lambda\coloneqq (1-\bar\theta)\xi$. 

The following result also holds for $\xi=0$.
\begin{Lemma}\label{asymptotic_pos_theta_ge1_pos_y}
  For $\theta >1$, $\xi\in[0,1]$ and $y\ge0$, the persistence probabilities $\left(p_n^{\xi,y}(\theta)\right)_{n\in\N_0}$ have a superexponential decay.
 \end{Lemma}
\begin{proof}
For $\theta>1$ and $y\ge 0$, we have, by Lemma~\ref{P_x_theta_positiv} and $[n]_\theta!=\frac{(\bar\theta;\bar\theta)_n}{(1-\bar\theta)^n}\theta^{\frac{n(n-1)}{2}}$,
    \begin{align*}
        p_n^{\xi,y}(\theta)
        &=c_ne^{-\alpha_ny}
        =\frac{(1-\xi)^n}{[n]_\theta!}e^{-\frac{\theta(1-\theta^n)}{1-\theta}y}
       =((1-\bar\theta)(1-\xi))^n\cdot \frac{1}{(\bar\theta;\bar\theta)_n}\cdot \theta^{-\frac{n(n-1)}{2}}\cdot e^{-\frac{\theta(1-\theta^n)}{1-\theta}y}.
    \end{align*}
The first term is an exponentially decaying factor. Since $\bar\theta<1$, we have $(\bar\theta;\bar\theta)_n\to (\bar\theta;\bar\theta)_\infty\in(0,\infty)$, so that the second term converges. As $\theta>1$, the third term decays superexponentially. The last term is bounded by $1$ (and decays superexponentially for $y>0$). This shows the claim.
\end{proof}
For the case $y\le 0$, abbreviate $\ell(z)\coloneqq\frac{1}{e_\theta(-\xi z)}$ and
\begin{align*}
  v(z)\coloneqq  v_y(z)\coloneqq  e_\theta((1-\xi)z)\sum_{j=0}^\infty \frac{(-\xi z)^{j}}{[j]_\theta!}e^{\frac{\theta(1-\theta^j)}{1-\theta}y}.
\end{align*}
\begin{Lemma}\label{v_entire_thetage1}
   For $\theta>1$, $\xi\in[0,1]$ and $y\le0$,  $v_y$ is entire.
\end{Lemma}
\begin{proof}
    By Lemma~\ref{e_theta_entire_thetage1}, $z\mapsto e_\theta((1-\xi)z)$ is entire and, for $y\le 0$ and $z\in\C$, it holds
\begin{align*}
    \left|\sum_{j=0}^\infty \frac{(-\xi z)^{j}}{[j]_\theta!}e^{\frac{\theta(1-\theta^j)}{1-\theta}y}\right|
    \le\sum_{j=0}^\infty \frac{|\xi z|^{j}}{[j]_\theta!}e^{\frac{\theta(1-\theta^j)}{1-\theta}y}
     \le\sum_{j=0}^\infty \frac{|\xi z|^{j}}{[j]_\theta!}
     =e_\theta(|\xi z|),
\end{align*}
so for every $R>0$, the series is uniformly dominated on the compact disk $|z|\le R$ by the convergent series $e_\theta(\xi R)$. In conclusion, $v_y$ is entire for all $y\le 0$.
\end{proof}
For the rest of this section, let $\xi> 0$, so $\lambda>0$. Define $h:\R\to\R$ by
\begin{align}\label{def_h_thetage1}
    h(y)\coloneqq h_{\theta,\xi}(y)\coloneqq  \begin{cases}
          0,~~~&y>0,
          \\ v_y\left(\frac{1}{\lambda}\right)={e_\theta\left(\frac{1-\xi}{\lambda}\right)}\sum_{j=0}^\infty \frac{\left(\frac{-\xi}{\lambda}\right)^{j}}{[j]_\theta!}e^{\frac{\theta(1-\theta^j)}{1-\theta}y},~~~~&y\le 0.
      \end{cases}
  \end{align}
  Evaluating $v_y$ at the candidate dominant pole gives the eigenfunction; the remaining key point is its positivity.
\begin{Lemma}\label{thetage1_h_pos}
$h$ is continuous and non-negative on $\R$ with $h(0)=0$ and $h(x)>0$ for all $x\in(-\infty,0)$.
\end{Lemma}

\begin{proof}
The fact that $h(0)=0$ follows directly from 
\begin{align*}
    h(0)={e_\theta\left(\frac{1-\xi}{\lambda}\right)}{e_\theta\left(\frac{-\xi}{\lambda}\right)},
\end{align*}
where ${e_\theta\left(\frac{-\xi}{\lambda}\right)}={e_\theta\left(\frac{-1}{ 1-\bar \theta}\right)}=0$ by Lemma~\ref{e_theta_entire_thetage1}. In particular, $h$ is continuous. To prove the positivity on the negative real line, note that
 ${e_\theta\left(\frac{1-\xi}{\lambda}\right)}>0$
since $\frac{1-\xi}{\lambda}\ge0$, $e_\theta(0)=1$ and by Lemma~\ref{e_theta_entire_thetage1}, the zeros of $e_\theta$ are negative. Hence, proving positivity of $h$ on the negative real line is equivalent to showing that for any $x>0$, the term 
\begin{align*}
    s(x)\coloneqq \frac {h(-x)}{e_\theta\left(\frac{1-\xi}{\lambda}\right)}=\sum_{j=0}^\infty \frac{\left(\frac{-\xi}{\lambda}\right)^{j}}{[j]_\theta!}e^{\frac{\theta(1-\theta^j)}{1-\theta}(-x)}=\sum_{j=0}^\infty \frac{(-1)^j}{ (1-\bar \theta)^j [j]_\theta!}\, e^{-\alpha_j x}
\end{align*}
is positive. To do so, we proceed as follows. First, we compute a differential equation for $s$. Then, we approximate $s$ with a positive function to get non-negativity. We conclude by combining these two steps to derive strict positivity.

Using $\alpha_j = \theta + \theta \alpha_{j-1}$ with $\alpha_0\coloneqq 0$, we get
\begin{align}\label{eqn:dgl}
s'(x)
 & = \sum_{j=1}^\infty \frac{(-1)^j}{ (1-\bar \theta)^j [j]_\theta!}\, (-\alpha_j) e^{-\alpha_j x}\notag
\\
 & = \frac{1}{1-\bar \theta}\sum_{j=1}^\infty \frac{(-1)^{j-1}}{ (1-\bar \theta)^{j-1} [j-1]_\theta! \frac{1-\theta^j}{1-\theta}}\, \frac{\theta (1-\theta^j)}{1-\theta} e^{-\theta x} e^{-\alpha_{j-1} \theta x}\notag
\\
 & = \frac{\theta}{1-\bar \theta}\,e^{-\theta x}\sum_{j=0}^{\infty} \frac{(-1)^{j}}{ (1-\bar \theta)^{j} [j]_\theta! }\,  e^{-\alpha_{j} \theta x}\notag
\\
&= \frac{\theta}{1-\bar \theta}\,e^{-\theta x} s(\theta x),
\end{align}
where the term-by-term differentiation is justified since $s'$ is locally uniformly convergent. Define approximations of $s$ by
\begin{align*}
S_n(x) & \coloneqq  \sum_{j=0}^n \frac{(-1)^j}{ (1-\bar \theta)^j [j]_\theta!}\,\frac{(\bar\theta;\bar\theta)_n}{(\bar\theta;\bar\theta)_{n-j}}\, e^{-\alpha_j x}
\end{align*}
and note that, for $n\ge 1$, using again $\alpha_j = \theta + \theta \alpha_{j-1}$,
\begin{align}
S_n'(x)
 & = \sum_{j=1}^n \frac{(-1)^j}{ (1-\bar \theta)^j [j]_\theta!}\, \frac{(\bar\theta;\bar\theta)_n}{(\bar\theta;\bar\theta)_{n-j}}\,(-\alpha_j) e^{-\alpha_j x} \notag
\\
 & = \frac{1}{1-\bar \theta}\sum_{j=1}^n \frac{(-1)^{j-1}}{ (1-\bar \theta)^{j-1} [j-1]_\theta! \frac{1-\theta^j}{1-\theta}}\, \frac{(1-\bar\theta^n)(\bar\theta;\bar\theta)_{n-1}}{(\bar\theta;\bar\theta)_{n-1-(j-1)}}\,\frac{\theta (1-\theta^j)}{1-\theta} e^{-\theta x} e^{-\alpha_{j-1} \theta x} \notag
\\
 & = \frac{\theta}{1-\bar \theta}\,e^{-\theta x}\sum_{j=0}^{n-1} \frac{(-1)^{j}}{ (1-\bar \theta)^{j} [j]_\theta! }\,  \frac{(1-\bar\theta^{n})(\bar\theta;\bar\theta)_{n-1}}{(\bar\theta;\bar\theta)_{n-1-j}}\, e^{-\alpha_{j} \theta x} \notag
\\
&= \frac{\theta(1-\bar\theta^{n})}{1-\bar \theta}\,e^{-\theta x} S_{n-1}(\theta x). \label{eqn:recursiondgl}
\end{align}
Now, by the Cauchy $q$-binomial theorem, for $n\ge 1$,
\begin{align}
S_n(0) & = \sum_{j=0}^n \frac{(-1)^j}{ (1-\bar \theta)^j [j]_\theta!}\,\frac{(\bar\theta;\bar\theta)_n}{(\bar\theta;\bar\theta)_{n-j}} \notag
\\
&= \sum_{j=0}^n \frac{(-1)^j(1-\theta)^j}{ (1-\bar \theta)^j (\theta;\theta)_j}\,\frac{(\bar\theta;\bar\theta)_n}{(\bar\theta;\bar\theta)_{n-j}}\notag
\\
&= \sum_{j=0}^n \frac{(-1)^j\bar\theta^{j(j-1)/2}(1-\bar\theta)^j}{(1-\bar \theta)^j} \,\frac{ (\bar\theta;\bar\theta)_n}{(\bar\theta;\bar\theta)_j(\bar\theta;\bar\theta)_{n-j}}\notag
\\
 &= \sum_{j=0}^n (-\theta)^j \bar\theta^{j(j+1)/2} \,\frac{ (\bar\theta;\bar\theta)_n}{(\bar\theta;\bar\theta)_j(\bar\theta;\bar\theta)_{n-j}}\notag
\\
&=\prod_{k=0}^{n-1}(1-\bar\theta^k)
 =0. \label{eqn:boundaryleft}
\end{align}
Further, as seen from the definition,
\begin{align} \label{eqn:boundaryright}
\lim_{x\to\infty} S_n(x) & = 1.
\end{align}
Now let us proceed by induction. For $n=0$, $S_0(x)=1$ for all $x>0$. In particular, $S_0(x)>0$ for all $x>0$. 
Assume that $S_{n-1}(x)>0$ for all $x>0$. Using (\ref{eqn:recursiondgl}), this shows that $S_n'(x)>0$ for all $x>0$. This and the boundary conditions (\ref{eqn:boundaryleft}) and (\ref{eqn:boundaryright}) show that we must have $S_n(x)>0$ for all $x>0$. 

Letting now $n\to\infty$, and using the dominated convergence theorem (with $\frac{(\bar\theta;\bar\theta)_n}{(\bar\theta;\bar\theta)_{n-j}}\le 1$ and observing that the remaining series is absolutely convergent) yields that
\begin{align*}
S_n(x) & = \sum_{j=0}^n \frac{(-1)^j}{ (1-\bar \theta)^j [j]_\theta!}\,\frac{(\bar\theta;\bar\theta)_n}{(\bar\theta;\bar\theta)_{n-j}}\, e^{-\alpha_j x} \to \sum_{j=0}^\infty \frac{(-1)^j}{ (1-\bar \theta)^j [j]_\theta!}\, e^{-\alpha_j x} = s(x),
\end{align*}
where we used the fact that $\frac{(\bar\theta;\bar\theta)_n}{(\bar\theta;\bar\theta)_{n-j}}=\prod_{i=n-j}^{n-1}(1-\bar\theta^{i+1}) \to 1$.

Since $S_n(x)>0$ for all $x>0$ and all $n$, we must have $s(x)\geq 0$ for all $x>0$.

Finally, the function $s$ satisfies $s(0)=0$ and, since $\alpha_j>0$ for every $j\ge1$, another application of
dominated convergence yields $\lim_{x\to\infty} s(x) = 1$. Moreover, we just proved that $s(x)\geq 0$ for all $x>0$. For the first property, note that $s(0)=0$ follows from $S_n(0)=0$ for all $n$.
From the fact that $s(x)\geq 0$, we see that $s$ must be a non-decreasing function, due to (\ref{eqn:dgl}). Assume that $s(x)=0$ for some $x>0$. Then, by monotonicity, $s(x')=0$ for all $0\leq x'\leq x$. This trivially implies that $s'(x')=0$ for all $0\leq x'< x$. Using the differential equation (\ref{eqn:dgl}), this shows that $s(x')=0$ for all $0\leq x' < \theta x$, i.e.\ a larger interval. Continuing this way, we deduce that $s(x)=0$ for all $x\geq 0$. This, however, contradicts the property $\lim_{x\to\infty} s(x) = 1$.
\end{proof}
With this positivity result, we can identify the dominant pole and derive the coefficient asymptotics.
\begin{Lemma}\label{gf_pos_theta_ge1}
     For $\theta>1$, $\xi\in(0,1]$ and $y< 0$, we have
    \begin{align*}
      \hat{P}_{\xi,y,\theta}(z)= \ell(z)v(z)
    \end{align*}
    with radius of convergence $\frac{1}{\lambda}$. More concretely, the unique singularity with smallest modulus of $\hat{P}_{\xi,y,\theta}(z)$ is simple, isolated and given by $\frac{1}{\lambda}$.
\end{Lemma}
\begin{proof}
By Lemma~\ref{gf_pos_theta}, $\hat{P}_{\xi,y,\theta}(z)= \ell(z)v(z)$ within the common radius of convergence, but at most within the radius of convergence of $R$. By Lemma~\ref{v_entire_thetage1}, $v$ is entire and, by Lemma~\ref{e_theta_entire_thetage1}, the pole of $\ell$ with smallest modulus is located at $\frac{1}{(1-\bar\theta)\xi}=\frac{1}{\lambda}$ and it is simple and isolated. Lemma~\ref{thetage1_h_pos} yields that $v\left(\frac{1}{\lambda}\right)=v_y\left(\frac{1}{\lambda}\right)=h(y)>0$, so $\frac{1}{\lambda}$ is indeed the unique singularity with smallest modulus of $\hat{P}_{\xi,y,\theta}(z)$. Here, we implicitly used the representation of $\hat{P}_{\xi,y,\theta}(z)$ from Lemma~\ref{gf_pos_theta}, which only holds within the radius of convergence of $R$. By the same argument as for $\hat{P}_{\xi,y,\theta}(z)$, the radius of convergence of $R$ is given by $\frac{1}{\lambda}$, so the claim follows.
\end{proof}
\begin{Lemma}\label{asymptotic_theta_ge1_neg_y}
  For $\theta>1$, $y\in\R$ and $\xi\in(0,1]$, we have, with $h$ defined in \eqref{def_h_thetage1},
  \begin{align*}
      \lim_{n\to\infty} \frac{p_n^{\xi,y}(\theta)\lambda(\bar\theta;\bar\theta)_\infty}{\lambda^{n+1}}=h(y).
  \end{align*}
\end{Lemma}
\begin{proof}
   For $y\ge0$, the claim follows by Lemma~\ref{asymptotic_pos_theta_ge1_pos_y} and $h(y)=0$. For $y<0$, we want to apply Lemma~\ref{asymptotic_powerseries}, so we check that there is a function $f$ with
    \begin{align}
        \hat{P}_{\xi,y,\theta}(z)=\frac{f(z)}{\frac{1}{\lambda}-z}~~~~\text{for all}~|z|<\frac{1}{\lambda}\label{eq_Lemma_2},
    \end{align}
   where $f$ is analytic on $|z|<z_1$ with $\frac{1}{\lambda}<z_1$ and $f\left(\frac{1}{\lambda}\right)\neq 0$. 
   By Lemma~\ref{gf_pos_theta_ge1}, $f(z)=\left(\frac{1}{\lambda}-z\right)\ell(z)v(z)$  is a suitable choice for \eqref{eq_Lemma_2}. Lemma~\ref{gf_pos_theta_ge1} yields that $\frac{1}{\lambda}$ is the unique singularity with smallest modulus of $\ell(z)v(z)$ and it is simple and isolated. In particular, there is some $z_1>\frac{1}{\lambda}$ such that the function $f$ is analytic on $|z|<z_1$ and by Lemma~\ref{thetage1_h_pos}, $h(y)=v_y\left(\frac{1}{\lambda}\right)>0$. The remaining part follows from the representation from \eqref{e_theta_prod_thetage1}
   \begin{align*}
      \left(\frac{1}{\lambda}-z\right)\ell(z)=\left(\frac{1}{\lambda}-z\right)\frac{1}{1-z\lambda}\prod_{n=1}^\infty\frac{1}{1-z\lambda \bar\theta^n}=\frac{1}{\lambda}\prod_{n=1}^\infty\frac{1}{1-z\lambda \bar\theta^n},
   \end{align*}
   which is equal to $\frac{1}{\lambda}\prod_{n=1}^\infty\frac{1}{1-\bar\theta^n}=\frac{1}{\lambda(\bar\theta;\bar\theta)_\infty}>0$ at $z=\frac{1}{\lambda}$, so $f\left(\frac{1}{\lambda}\right)=\frac{1}{\lambda(\bar\theta;\bar\theta)_\infty}h(y)>0$.
  Thus, $f$ satisfies the  conditions for Lemma \ref{asymptotic_powerseries}, which now yields $p_n^{\xi,y}(\theta)\sim \frac{1}{\lambda(\bar\theta;\bar\theta)_\infty}h(y)\lambda^{n+1}$.
\end{proof}
\begin{proof}[Proof of Theorem~\ref{maintheorem} (ii)]
 Equation (5) and Theorem 1 (a) of \cite{AR}
 imply that for $n\to\infty$, $p_n^{\xi}(\theta)\sim c_2(\theta,\xi)\lambda^{n+1}$ for some constant $c_2(\theta,\xi)>0$,
       so Lemma~\ref{asymptotic_theta_ge1_neg_y} together with Lemma~\ref{eigenfunction_lemma} and Lemma~\ref{Kernel_lim_process} yield the claim.
\end{proof}
\subsection{Invariant distributions}\label{invariant distributions}
With the eigenvalue equation \eqref{eq_eigenfunction} shown in Theorem~\ref{maintheorem} (i) and (ii), we can now deduce the existence of the invariant distribution claimed in Theorem~\ref{invariant_distribution}.
\begin{proof}[Proof of Theorem~\ref{invariant_distribution} (i)]
Let $0<\theta<1$ and $\xi\in[0,1)$. Define $\bar\theta\coloneqq \theta^{-1}$ and
\begin{align*} 
    \tilde h_{\theta,\xi}(x) \coloneqq \mathds{1}_{\{x>0\}} h_{\bar\theta,\,1-\xi}(-x), \qquad x\in\R. 
\end{align*} 
We first verify that $\tilde h_{\theta,\xi}$ satisfies the left eigenfunction equation \eqref{eq_left_eigenfunction} from Lemma~\ref{invariant_distribution_limiting_process}. Recall that reflection interchanges the two sides of the asymmetric Laplace distribution. Thus, for every non-negative measurable function $f$, 
\begin{align*}
    \int_\R f(-x)\phi_\xi( x)\d x = \int_\R f(z)\phi_{1-\xi}( z)\d z. 
\end{align*} 
Let $y>0$. Since $\theta>0$, we have $y>\theta x $ if and only if $x<\bar\theta y$. 
Consequently, using the substitution $z=-x$, we obtain
 \begin{align*} 
    \int_\R \mathds{1}_S(y,x) \tilde h_{\theta,\xi}(x) \phi_\xi( x)\d x 
     &= \int_0^{\bar\theta y} h_{\bar\theta,\,1-\xi}(-x) \phi_\xi( x)\d x 
    \\ &= \int_{\R_{<0}} \mathds{1}_{\{z>\bar\theta(-y)\}} h_{\bar\theta,\,1-\xi}(z) \phi_{1-\xi}( z)\d z 
    \\ &= \lambda_{\bar\theta,\,1-\xi} h_{\bar\theta,\,1-\xi}(-y)
    \\ &= \lambda_{\theta,\xi} \tilde h_{\theta,\xi}(y), 
\end{align*}
where the last but one step follows from the right eigenfunction equation for the parameters $\bar\theta>1$ and $1-\xi$ from Theorem~\ref{maintheorem} (ii) and the last equality holds because corresponding eigenvalues agree:
\begin{align} 
    \lambda_{\bar\theta,\,1-\xi} 
    &= \left( 1-\frac{1}{\bar\theta} \right)(1-\xi) 
    = (1-\theta)(1-\xi) 
    = \lambda_{\theta,\xi}. \label{eigenvalue_identity}
\end{align} 
If $y\leq0$, then $x>0$ from the definition of $\tilde h_{\theta,\xi}$ and $y>\theta x$ cannot hold simultaneously. Hence, both sides of the left eigenfunction equation vanish. We have thus shown that 
\begin{align*} 
    \int_\R \mathds{1}_S(y,x) \tilde h_{\theta,\xi}(x) \phi_\xi( x)\d x = \lambda_{\theta,\xi} \tilde h_{\theta,\xi}(y), \qquad y\in\R. 
\end{align*} 
Set 
\begin{align*} 
    I  \coloneqq \int_\R h_{\theta,\xi}(x) \tilde h_{\theta,\xi}(x) \phi_\xi( x)\d x. 
\end{align*}
Since the two eigenfunctions and $\phi_\xi$ are strictly positive on $(0,\infty)$, we have $I >0$. Moreover, the persistence asymptotics from Lemma~\ref{asymptotic_pos_theta_pos_y}, Lemma~\ref{asymptotic_pos_theta_neg_y} and Lemma~\ref{asymptotic_theta_ge1_neg_y} applied with \eqref{eigenvalue_identity} together with Lemma~\ref{estimate_Px} and equation (5) from \cite{AR} imply that both $h_{\theta,\xi}$ and $h_{\bar\theta,\,1-\xi}$ are bounded on their respective state spaces. Since $\phi_\xi$ is a probability density, it follows that $I <\infty$.
Lemma~\ref{invariant_distribution_limiting_process} now shows that the following probability measure is invariant for $p'$:
\begin{align*} 
    \pi (\d x) 
    = \frac{ h_{\theta,\xi}(x) \tilde h_{\theta,\xi}(x) }{ I  } \phi_\xi( x)\d x 
    = \frac{ h_{\theta,\xi}(x) h_{\bar\theta,\,1-\xi}(-x) }{ I  } \mathds{1}_{\{x>0\}} \phi_\xi( x)\d x.
\end{align*}

Now, we identify this invariant distribution. Instead of the following proof, one could also plug in $h_{\theta,\xi}$ from \eqref{def_h_pos_theta_2} and $h_{\bar\theta,\,1-\xi}$ from \eqref{def_h_thetage1} and transform the term above by a direct calculation. However, we will do a shorter proof applying a fixed-point argument to an autoregressive representation of the conditioned chain. To do so, note that  $\pi$ is concentrated on $(0,\infty)$ and $(0,\infty)$ is closed under the transition kernel $p'$. On this state space, the following holds. For $x>0$, the right eigenfunction and eigenvalue are given by 
\begin{align*} 
    h_{\theta,\xi}(x) = \exp\left( -\frac{\theta}{1-\theta}x \right) 
\end{align*} 
and $\lambda_{\theta,\xi} = (1-\xi)(1-\theta)$.
Moreover, $y>\theta x>0$ implies $ \phi_\xi(y) = (1-\xi)e^{-y}$.
Thus, 
\begin{align*} 
    p' (x,\d y) 
    &= \mathds{1}_{\{y>\theta x\}} \frac{ \exp\left( -\frac{\theta}{1-\theta}y \right) }{ (1-\xi)(1-\theta) \exp\left( -\frac{\theta}{1-\theta}x \right) } (1-\xi)e^{-y}\d y 
    \\ &= \frac{1}{1-\theta} \exp\left( -\frac{y-\theta x}{1-\theta} \right) \mathds{1}_{\{y>\theta x\}} \d y. 
\end{align*} 
Hence, on $(0,\infty)$, the conditioned chain satisfies the autoregressive equation $Y_{n+1} = \theta Y_n+(1-\theta)E_{n+1}$, 
where $(E_n)_{n\geq1}$ are i.i.d. exponentially distributed random variables with parameter $1$. To find a fixed point of this equation, first note that the random variable 
\begin{align*} 
    Y_\infty \coloneqq (1-\theta) \sum_{k=0}^\infty \theta^kE_{k+1} 
\end{align*} 
is well defined almost surely since $\E[Y_\infty]=(1-\theta)\sum_{k=0}^\infty \theta^k=1$ and the summands are non-negative. If $E_0$ is another exponentially distributed random variable independent of $Y_\infty$, then 
\begin{align*} 
    (1-\theta)E_0+\theta Y_\infty 
    = (1-\theta)E_0 + (1-\theta) \sum_{k=0}^\infty \theta^{k+1}E_{k+1} 
    \overset{\d}{=} (1-\theta) \sum_{k=0}^\infty \theta^kE_{k+1} 
    = Y_\infty. 
\end{align*} 
Thus, the law of $Y_\infty$ is invariant for the restriction of $p'$ to $(0,\infty)$. 

To conclude, we show that the law of $Y_\infty$ is the unique invariant distribution of the restricted kernel. Let $\nu$ be an invariant probability measure for the restriction of $p'$ to $(0,\infty)$, and let $Y_0$ have distribution $\nu$ and be independent of $(E_n)_{n\geq1}$. Then
\begin{align*}
    Y_n  = \theta^nY_0  +(1-\theta)\sum_{j=1}^n\theta^{n-j}E_j.
\end{align*}
Since $0<\theta<1$, we have $\theta^nY_0\to0$ almost surely. Moreover,
\begin{align*}
    (1-\theta)\sum_{j=1}^n\theta^{n-j}E_j  \overset{\d}{=}(1-\theta)\sum_{k=0}^{n-1}\theta^kE_{k+1},
\end{align*}
where the right-hand side converges to $Y_\infty$ almost surely. Hence, by Slutsky's theorem, $Y_n$ converges to $Y_\infty$ in distribution. On the other hand, the invariance of $\nu$ implies that $Y_n$ has distribution $\nu$ for every $n\in\N_0$. Consequently, $\nu $ is given by the law of $Y_\infty$ and thus, the restriction of $p'$ to $(0,\infty)$ has the unique invariant distribution claimed in \eqref{inv_distr_i}.

It remains to prove uniqueness on the full state space $\R$. Let $\mu$ be an arbitrary invariant probability measure for $p'$ on $\R$ and set $A\coloneqq(0,\infty)$. Since $p'(x,A)=1$ for every $x\in A$, the invariance of $\mu$ on $\R$ yields
\begin{align*}
    \mu(A)
    &= \int_{\R}p'(x,A)\mu(\d x) 
    = \int_Ap'(x,A)\mu(\d x)
       +\int_{A^c}p'(x,A)\mu(\d x) 
    \\&= \mu(A)+\int_{A^c}p'(x,A)\mu(\d x).
\end{align*}
For every $x\in A^c$, the transition density of $p'$ is strictly positive on $A$, and hence $p'(x,A)>0$. Consequently, $\mu(A^c)=0$. Thus, every invariant probability measure for $p'$ is concentrated on $(0,\infty)$ and must coincide with the unique invariant distribution $\pi$ of the restricted kernel. Hence, $\pi$ is also the unique invariant distribution on $\R$.
\end{proof}
\begin{proof}[Proof of Theorem~\ref{invariant_distribution} (ii)]
Let $\theta>1$ and $\xi\in(0,1]$ and recall $
\bar\theta \coloneqq \frac{1}{\theta}\in(0,1)$.
Throughout this proof, we consider the eigenvalue problem and the
conditioned kernel restricted to the state space $\R_{<0}$. Define 
\begin{align*} 
    \tilde h_{\theta,\xi}(x) \coloneqq h_{\bar\theta,1-\xi}(-x), \qquad x\in\R_{<0}. 
\end{align*} 
Let $y<0$. Using the substitution $z=-x$ and the reflection property of the asymmetric Laplace distribution, we obtain
\begin{align*} 
    \int_{\R_{<0}} \mathds{1}_S(y,x) \tilde h_{\theta,\xi}(x) \phi_\xi( x)\d x 
     &= \int_{-\infty}^{\bar\theta y} h_{\bar\theta,1-\xi}(-x) \phi_\xi( x)\d x 
    \\ &= \int_\R \mathds{1}_{\{z>\bar\theta(-y)\}} h_{\bar\theta,1-\xi}(z) \phi_{1-\xi}( z)\d z 
    \\ &= \lambda_{\bar\theta,1-\xi} h_{\bar\theta,1-\xi}(-y)
        \\ &= \lambda_{\theta,\xi} \tilde h_{\theta,\xi}(y),
\end{align*} 
where we used the right eigenvalue equation \eqref{eq_eigenfunction} from Theorem~\ref{maintheorem} (i) and
\begin{align*} 
    \lambda_{\bar\theta,1-\xi}
    = (1-(1-\xi))(1-\bar\theta) 
    = \xi\left(1-\frac{1}{\theta}\right) 
    = \lambda_{\theta,\xi}. 
\end{align*} 
It follows that 
\begin{align*} 
    \int_{\R_{<0}} \mathds{1}_S(y,x) \tilde h_{\theta,\xi}(x) \phi_\xi( x)\d x = \lambda_{\theta,\xi} \tilde h_{\theta,\xi}(y), \qquad y<0. 
\end{align*}
Set 
\begin{align*} 
    I\coloneqq \int_{-\infty}^0 h_{\theta,\xi}(x) h_{\bar\theta,1-\xi}(-x) \phi_\xi( x)\d x.
\end{align*} 
The eigenfunctions and $\phi_\xi$ are strictly positive on the relevant state spaces, so $I >0$. Their boundedness follows from Lemma~\ref{estimate_Px} and the corresponding persistence asymptotics from Lemma~\ref{asymptotic_pos_theta_pos_y} and Lemma~\ref{asymptotic_theta_ge1_neg_y} as well as Equation (5) of \cite{AR}. Hence, $I <\infty$.
Lemma~\ref{invariant_distribution_limiting_process} gives that
\begin{align*} 
    \pi_{\theta,\xi} (\d x) = \frac{ h_{\theta,\xi}(x) h_{\bar\theta,1-\xi}(-x) }{ I  } \phi_\xi( x)\d x, \qquad x<0,
\end{align*} 
is invariant for $p'$. It remains to identify this distribution. Let $\pi_{\bar\theta,1-\xi}$ denote the invariant probability measure from Theorem~\ref{invariant_distribution}~(i). The density representation
obtained in its proof has normalization constant
\begin{align*}
    \bar I
    \coloneqq\int_0^\infty h_{\bar\theta,1-\xi}(z)h_{\theta,\xi}(-z)\phi_{1-\xi}(z)\d z
    =\int_{-\infty}^0h_{\bar\theta,1-\xi}(-x)h_{\theta,\xi}(x)\phi_\xi(x)\d x
    =I,
\end{align*}
where we used $z=-x$ and
$\phi_{1-\xi}(-x)=\phi_\xi(x)$. 
Further,
\begin{align*}
    \pi_{\theta,\xi}(\d x) 
    &=\frac{1}{I}\, h_{\theta,\xi}(x)h_{\bar\theta,1-\xi}(-x)\phi_\xi(x) 
    =\frac{1}{\bar I} \, h_{\bar\theta,1-\xi}(-x)h_{\theta,\xi}(x)\phi_{1-\xi}(-x)
    =\pi_{\bar\theta,1-\xi}(-\d x).
\end{align*}
By Theorem~\ref{invariant_distribution}~(i),
\begin{align*}
    \pi_{\bar\theta,1-\xi}(\d x)
    =
    \P\left(
        (1-\bar\theta)
        \sum_{k=0}^\infty\bar\theta^kE_{k+1}
        \in\d x
    \right).
\end{align*}
Consequently,
\begin{align*}
    \pi_{\theta,\xi}(\d x)
    =
    \P\left(
        -(1-\theta^{-1})
        \sum_{k=0}^\infty\theta^{-k}E_{k+1}
        \in\d x
    \right).
\end{align*}
Finally, uniqueness follows from irreducibility. Since $\xi>0$ and $h_{\theta,\xi}$ is strictly positive on $\R_{<0}$, the reduced transition density of $p'$ is strictly positive on $(\theta x,0)$ for every $x<0$. If $p^n(x,\cdot)$ denotes the $n$-step transition density of $p'$, an induction using the Chapman-Kolmogorov equation yields
\begin{align*}
p^n(x,y)>0,\qquad x<0,\quad y\in(\theta^n x,0),\quad n\in\N.
\end{align*}
Since $\theta^n x\to-\infty$, every Borel set of positive Lebesgue measure can be reached from every $x<0$ with positive probability in finitely many steps. Thus, $p'$ is Lebesgue-irreducible on $\R_{<0}$. By the standard uniqueness result for irreducible Markov chains, $p'$ admits at most one invariant probability measure and since $\pi_{\theta,\xi}$ is invariant, it is the unique invariant distribution.
\end{proof}
\section{The case \texorpdfstring{$\theta<0$}{theta<0}}\label{proofs_theta_neg}
The special cases $\xi=0$ and $\xi=1$ are treated in Section~\ref{trivialcases}. For the remainder of this section, consider $\xi\in(0,1)$ and abbreviate 
\begin{align*}
    \gamma\coloneqq\frac{1-\xi}{\xi}> 0\quad\text{and}\quad u\coloneqq (\xi-1)\gamma^{-1/2}=-\sqrt{\xi(1-\xi)}=-\xi \gamma^{1/2}.
\end{align*}
In particular, the value $\lambda\coloneqq \frac{-u}{\arctan_{-\theta}\left(\gamma^{-1/2}\right)}>\frac{-u}{z_c}$, where $z_c\coloneqq z_c(-\theta)$ is well-defined by Corollary~\ref{arctan_cor}.

\subsection{Computation of the generating functions}
Lemma 9 in \cite{AR} provides an explicit formula for the persistence probability $p_n^{\xi,y}(\theta)$:
\begin{Lemma}\label{pers_prob_neg_theta}
    Let $\theta<0$ and $\xi\in(0,1)$. For all $n\geq1,$, it holds
        \begin{align*}
        p_n^{\xi,y}(\theta)= \begin{cases}
        \sum_{j=1}^n\frac{(\xi-1)^{j-1}r^n_{n-j+1}}{[j]_{-\theta}!}e^{-\alpha_jy}, & y\leq 0,
       \\
       \sum_{j=1}^n\frac{(-\xi)^{j-1}s^n_{n-j+1}}{[j]_{-\theta}!}(1-e^{\alpha_jy})+q_n, & y\geq 0,
       \end{cases}
    \end{align*}
where $\alpha_1\coloneqq\theta$, $r_1^1\coloneqq 1-\xi$, $s_1^1\coloneqq \xi$, $q_1\coloneqq 1-\xi$ and the further constants for $n\ge 2$ are defined via the recursions
\begin{align*}
    \alpha_n&\coloneqq (1-\alpha_{n-1})\theta=-\sum_{i=1}^n(-\theta)^i,
    \\s_n^n&\coloneqq 0,
  \\ r_{i}^n&\coloneqq\left( \frac{\xi}{1-\xi}\right)^{n-i-1} s_{i}^{n-1}, \quad s_i^n\coloneqq(-1)\left( \frac{1-\xi}{\xi}\right)^{n-i-1} r_i^{n-1},\quad i=1,\ldots,n-1,
 \\ r_n^n&\coloneqq\sum_{j=1}^{n-1} \frac{(-1)^{j-1} \xi^{j-1} (1-\xi) s_{n-j}^{n-1}}{\prod_{m=1}^{j-1} (1-\alpha_m)}+ q_{n-1}(1-\xi),~~~\text{and}
 \\q_n&\coloneqq r_n^n- \sum_{j=1}^{n-1} \frac{(-1)^{j-1}\xi^{j-1} (1-\xi) s_{n-j}^{n-1}}{\prod_{m=1}^{j} (1-\alpha_m)}.
\end{align*}
\end{Lemma}
\begin{Remark}
    Note that $[j]_{-\theta}!=\prod_{m=1}^{j-1} (1-\alpha_m)$ and for $\theta\neq-1$, $\alpha_n=\frac{\theta(1-(-\theta)^n)}{1+\theta}$. For $\theta=-1$, we have $\alpha_n=-n$.
\end{Remark}
In the following, we give formulas for the following power series: Define
\begin{align*}
    S^j(z)&\coloneqq\sum_{n=1}^\infty s_n^{n+j-1}z^n,\quad R^j(z)\coloneqq\sum_{n=1}^\infty r_n^{n+j-1}z^n,~~~j\in\N,~\text{and}\qquad
    Q(z)\coloneqq\sum_{n=1}^\infty q_nz^n.
\end{align*}
 From the proof of Theorem 2(h) in \cite{AR} and the symmetry of $\cos_\theta$ and the anti-symmetry of $\sin_\theta$, we deduce the following identities.
\begin{Lemma}\label{gf_other_neg_theta}
   Let $\theta<0$ and $\xi\in(0,1)$. We have
    \begin{align*}
        R^1(z)&=\frac{(1-\xi)z(\cos_{-\theta}(\xi\gamma^{1/2}z)+\gamma^{-1/2}\sin_{-\theta}(\xi\gamma^{1/2}z))}{\cos_{-\theta}(\xi\gamma^{1/2}z)-\gamma^{1/2}\sin_{-\theta}(\xi\gamma^{1/2}z)}
        \\&=\frac{(1-\xi)z(\cos_{-\theta}(uz)-\gamma^{-1/2}\sin_{-\theta}(uz))}{\cos_{-\theta}(uz)+\gamma^{1/2}\sin_{-\theta}(uz)}, 
        \\R^j(z)&=\gamma^{-(j-2)}S^{j-1}(z),~~~j\ge2,
        \\Q(z)&=\xi^{-1}\gamma^{-1/2}z^{-1}\sin_{-\theta}(\xi\gamma^{1/2}z) R^1(z)+\cos_{-\theta}(\xi\gamma^{1/2}z)-1
        \\&=-\xi^{-1}\gamma^{-1/2}z^{-1}\sin_{-\theta}(uz) R^1(z)+\cos_{-\theta}(uz)-1\qquad \text{and}
\\
    S^j(z)& =\begin{cases}
        (-1)^{j/2}\gamma^{(j-2)/2}R^1(z),~~~&j\ge 2\text{ even},
        \\(-1)^{(j-1)/2}\gamma^{(j-1)/2}\xi z,~~~&j\ge 1\text{ odd}.
    \end{cases}
\end{align*}
\end{Lemma}
\begin{Remark}\label{R1_analytic_0}
    By definition, the function $z\mapsto \frac{R^1(z)}{(1-\xi) z}$ has an analytic extension in $z=0$ with value $1$.
\end{Remark}
\begin{Corollary}\label{Rj_neg_theta}
  For $\theta<0$ and $\xi\in(0,1)$, it holds
    \begin{align*}
        R^j(z)=\begin{cases}
        (-1)^{(j-2)/2}\gamma^{(j-2)/2-(j-2)}\xi z,~~~&j\ge 2\text{ even},
        \\(-1)^{(j-1)/2}\gamma^{(j-3)/2-(j-2)}R^1(z),~~~&j\ge 1\text{ odd}.
        \end{cases}
    \end{align*}
\end{Corollary}
\begin{proof}
    The proof follows immediately from Lemma~\ref{gf_other_neg_theta}.
\end{proof}
We now combine the coefficient formulas with the preceding generating-function identities to obtain $\hat{P}_{\xi,y,\theta}$.
\begin{Lemma}\label{gf_neg_theta}
Let $\theta<0$, $\theta\neq-1$, $\xi\in(0,1)$, and $y\in\R$.
Choose $\epsilon=\epsilon(\xi,y,\theta)>0$ such that all power series
occurring below converge absolutely for $|z|<\epsilon$. Then, for
$|z|<\epsilon$,
    \begin{align*}
      \hat{P}_{\xi,y,\theta}(z)=  
      \begin{cases}
          \gamma^{-1/2}\sin_{-\theta}(uz,-y)+\frac{R^1(z)}{(1-\xi) z}\cos_{-\theta}(uz,-y),~~~&y>0,
          \\\cos_{-\theta}(uz,y)
     -\gamma^{1/2}\frac{R^1(z)}{(1-\xi) z}\sin_{-\theta}(uz,y),&y\le 0,
      \end{cases}
    \end{align*}
    where $\frac{R^1(z)}{(1-\xi) z}=\frac{\cos_{-\theta}(uz)-\gamma^{-1/2}\sin_{-\theta}(uz)}{\cos_{-\theta}(uz)+\gamma^{1/2}\sin_{-\theta}(uz)}$.
\end{Lemma}
\begin{proof}
  First, note that such an $\epsilon>0$ exists because $p_n^{\xi,y}(\theta)\leq1$ implies absolute convergence of $\hat P_{\xi,y,\theta}$ for $|z|<1$, while all $q$-trigonometric functions on the right-hand sides are holomorphic in a neighborhood of $0$ and 
  the singularity of the quotient in Remark~\ref{R1_analytic_0} at $z=0$ can be removed and the function can be extended analytically.
  
  Let $y\le0$. For $|z|<\epsilon$, by Lemma~\ref{pers_prob_neg_theta} and Corollary~\ref{Rj_neg_theta}, we have
  \begin{align*}
     \hat{P}_{\xi,y,\theta}(z)-1
     &=\sum_{n=1}^{\infty}p_n^{\xi,y}(\theta)z^n
     =\sum_{n=1}^{\infty}\sum_{j=1}^n\frac{(\xi-1)^{j-1}r^n_{n-j+1}}{[j]_{-\theta}!}e^{-\alpha_jy}z^n
     \\&=\sum_{j=1}^\infty\frac{(\xi-1)^{j-1}R^j(z)}{[j]_{-\theta}!}e^{-\alpha_jy}z^{j-1}
      \\&=\sum_{j=1}^\infty\frac{(z(\xi-1))^{2j-1}(-1)^{j-1}\gamma^{1-j}\xi z}{[2j]_{-\theta}!}e^{-\alpha_{2j}y}
     \\&~~~~
     +\sum_{j=0}^\infty\frac{(z(\xi-1))^{2j}(-1)^{j}\gamma^{-j}R^1(z)}{[2j+1]_{-\theta}!}e^{-\alpha_{2j+1}y}
     \\&=\frac{-\gamma \xi }{ \xi-1}\sum_{j=1}^\infty\frac{(-1)^{j}(z(\xi-1)\gamma^{-1/2})^{2j}}{[2j]_{-\theta}!}e^{-\frac{\theta(1-\theta^{2j})}{1+\theta}y}
     \\&~~~~+R^1(z)\sum_{j=0}^\infty\frac{(-1)^{j}(z(\xi-1)\gamma^{-1/2})^{2j}}{[2j+1]_{-\theta}!}e^{-\frac{\theta(1+\theta^{2j+1})}{1+\theta}y}
      \\&=\cos_{-\theta}(uz,y)-1
     +\gamma^{1/2}\frac{R^1(z)}{(\xi-1)z}\sin_{-\theta}(uz,y),
  \end{align*}
  where we used $\frac{-\gamma \xi z}{ z(\xi-1)}=1$ in the last step.  For $y>0$ and  $|z|<\epsilon$, we have, using Remark~\ref{R1_analytic_0},
   \begin{align*}
     \hat{P}_{\xi,y,\theta}(z)-1
     &=\sum_{n=1}^{\infty}p_n^{\xi,y}(\theta)z^n
     =\sum_{n=1}^{\infty}\sum_{j=1}^n\frac{(-\xi)^{j-1}s^n_{n-j+1}}{[j]_{-\theta}!}(1-e^{\alpha_jy})z^n+\sum_{n=1}^{\infty}q_nz^n
     \\&=\sum_{j=1}^\infty\frac{(-\xi z)^{j-1}S^j(z)}{[j]_{-\theta}!}(1-e^{\alpha_jy})+Q(z)
     \\&=\sum_{j=1}^\infty\frac{(-\xi z)^{2j-1}(-1)^{j}\gamma^{j-1}R^1(z)}{[2j]_{-\theta}!}(1-e^{\alpha_{2j}y})
     \\&~~~~+\sum_{j=0}^\infty\frac{(-\xi z)^{2j}(-1)^{j}\gamma^{j}\xi z}{[2j+1]_{-\theta}!}(1-e^{\alpha_{2j+1}y})+Q(z) 
     \\&=\frac{R^1(z)}{(\xi-1) z}\sum_{j=1}^\infty\frac{(-1)^{j}(-\xi z\gamma^{1/2})^{2j}}{[2j]_{-\theta}!}(1-e^{\frac{\theta(1-\theta^{2j})}{1+\theta}y})
     \\&~~~~+\xi z\sum_{j=0}^\infty\frac{(-1)^{j}(-\xi z\gamma^{1/2})^{2j}}{[2j+1]_{-\theta}!}(1-e^{\frac{\theta(1+\theta^{2j+1})}{1+\theta}y})+Q(z)
     \\&=\frac{R^1(z)}{(\xi-1) z}(\cos_{-\theta}\left(uz\right)-1)-\frac{R^1(z)}{(\xi-1) z}\sum_{j=1}^\infty\frac{(-1)^{j}(uz)^{2j}}{[2j]_{-\theta}!}e^{\frac{\theta(1-\theta^{2j})}{1+\theta}y}
     \\&~~~~-\gamma^{-1/2}\sin_{-\theta}(uz)-\xi z\sum_{j=0}^\infty\frac{(-1)^{j}(uz)^{2j}}{[2j+1]_{-\theta}!}e^{\frac{\theta(1+\theta^{2j+1})}{1+\theta}y}+Q(z)
      \\&=\frac{R^1(z)}{(\xi-1) z}\cos_{-\theta}\left(uz\right)-\frac{R^1(z)}{(\xi-1) z}\sum_{j=0}^\infty\frac{(-1)^{j}(uz)^{2j}}{[2j]_{-\theta}!}e^{\frac{1-\theta^{2j}}{1+\theta}\theta y}
     \\&~~~~-\gamma^{-1/2}\sin_{-\theta}(uz)-\xi z\sum_{j=0}^\infty\frac{(-1)^{j}(uz)^{2j}}{[2j+1]_{-\theta}!}e^{\frac{1+\theta^{2j+1}}{1+\theta}\theta y}+Q(z).
  \end{align*}
Using Lemma~\ref{gf_other_neg_theta}, this gives
 \begin{align*}
     \hat{P}_{\xi,y,\theta}(z)-1
     &=\cos_{-\theta}\left(uz\right)\left(\frac{R^1(z)}{(\xi-1) z}+1\right)
     -\gamma^{-1/2}\sin_{-\theta}\left(uz\right)\left(\frac{R^1(z)}{\xi z}+1\right)-1
     \\&~~~~~-\frac{R^1(z)}{(\xi-1) z}\cos_{-\theta}(uz,-y)
     +\gamma^{-1/2}\sin_{-\theta}(uz,-y).
    \end{align*}
With 
\begin{align*}
    \frac{R^1(z)}{(\xi-1) z}+1
    &=\frac{(\gamma^{1/2}+\gamma^{-1/2})\sin_{-\theta}(uz)}{\cos_{-\theta}(uz)+\gamma^{1/2}\sin_{-\theta}(uz)}, \quad \frac{R^1(z)}{\xi z}+1
    =\frac{(1+\gamma)\cos_{-\theta}(uz)}{\cos_{-\theta}(uz)+\gamma^{1/2}\sin_{-\theta}(uz)},
\end{align*}
we get 
\begin{align*}
    \hat{P}_{\xi,y,\theta}(z)
     &=\cos_{-\theta}\left(uz\right)\frac{(\gamma^{1/2}+\gamma^{-1/2})\sin_{-\theta}(uz)}{\cos_{-\theta}(uz)+\gamma^{1/2}\sin_{-\theta}(uz)}-\frac{R^1(z)}{(\xi-1) z}\cos_{-\theta}(uz,-y)
     \\&~~~~~-\gamma^{-1/2}\sin_{-\theta}\left(uz\right)\frac{(1+\gamma)\cos_{-\theta}(uz)}{\cos_{-\theta}(uz)+\gamma^{1/2}\sin_{-\theta}(uz)}
     +\gamma^{-1/2}\sin_{-\theta}(uz,-y)
     \\&=\frac{(\gamma^{1/2}+\gamma^{-1/2})\sin_{-\theta}(uz)\cos_{-\theta}(uz)-\gamma^{-1/2}(1+\gamma)\sin_{-\theta}(uz)\cos_{-\theta}(uz)}{\cos_{-\theta}(uz)+\gamma^{1/2}\sin_{-\theta}(uz)}
     \\&~~~~~+\gamma^{-1/2}\sin_{-\theta}(uz,-y)-\frac{R^1(z)}{(\xi-1) z}\cos_{-\theta}(uz,-y)
     \\&=\gamma^{-1/2}\sin_{-\theta}(uz,-y)-\frac{R^1(z)}{(\xi-1) z}\cos_{-\theta}(uz,-y).\tag*{\qedhere}
\end{align*}
\end{proof}
\begin{Remark}
    The identities in Lemma~\ref{gf_neg_theta} are initially obtained
for $|z|<\epsilon$. Since the expressions on their right-hand
sides are meromorphic in $z$, they provide meromorphic
continuations of $\hat P_{\xi,y,\theta}$. Hence, the
singularities of these continuations can be studied independently
of $\epsilon$. The radius of convergence $\delta$ of
$\hat P_{\xi,y,\theta}$ will be determined below.
\end{Remark}
The following is the version of Lemma~\ref{gf_neg_theta} for $\theta=-1$. It follows directly from the calculation in the proof of Lemma~\ref{gf_neg_theta} together with $\alpha_n=-n$.
\begin{Corollary}\label{gf_theta-1}
    For $\theta=-1$ and $\xi\in(0,1)$, it holds
    \begin{align*}
      \hat{P}_{\xi,y,\theta}(z)=  
      \begin{cases}
          \gamma^{-1/2}\sin(uze^{-y})+\frac{R^1(z)}{(1-\xi) z}\cos(uze^{-y}),~~~&y>0,
          \\\cos(uze^y)
     -\gamma^{1/2}\frac{R^1(z)}{(1-\xi) z}\sin(uze^y),&y\le 0,
      \end{cases}
    \end{align*}
where $\frac{R^1(z)}{(1-\xi) z}=\frac{\cos(uz)-\gamma^{-1/2}\sin(uz)}{\cos(uz)+\gamma^{1/2}\sin(uz)}$.
\end{Corollary}
\subsection{Singularities of the generating functions}
Before locating the dominant singularity, we show that the radius of convergence does not depend on the starting point.
\begin{Lemma}\label{roc_P_not_depending_on_y}
     For $\theta<0$ and $\xi\in(0,1)$, the radius of convergence $\delta$ of $\hat{P}_{\xi,y,\theta}(z)$ does not depend on $y$. In particular, $\hat{P}_{\xi,y,\theta}(z)$ has the same radius of convergence as $\hat{P}_{\xi,0,\theta}(z)$.
\end{Lemma}
\begin{proof}
    To simplify notation, abbreviate $p_n(x)\coloneqq p_n^{\xi,x}(\theta)$ and $\hat{P}_x(z)\coloneqq \hat{P}_{\xi,x,\theta}(z)$. Recall that $x\mapsto p_n(x)$ is monotonically increasing for all $n\in\N_0$ since $\theta<0$ and 
    \begin{align}\label{proof_roc}
        p_{n+1}(x)=\int_{\{t>\theta x\}\cap(-\infty,0)}\xi  e^tp_n(t)\d t+\int_{\{t>\theta x\}\cap(0,\infty)}(1-\xi)  e^{-t}p_n(t)\d t.
    \end{align}
Let $y>0$. By monotonicity, $p_n(y)\geq p_n(0)$ and by \eqref{proof_roc},
\begin{align*}
   p_{n+1}(0)&= \int_{(0,\infty)}(1-\xi)  e^{-t}p_n(t)\d t
   \ge (1-\xi)\int_{(y,\infty)}  e^{-t}p_n(t)\d t
   \ge (1-\xi)e^{-y}p_n(y),
\end{align*}
so $p_n(0)\le p_n(y)\le\frac{e^y}{1-\xi}p_{n+1}(0)$. In particular, for any $r>0$, it holds $\hat{P}_0(r)\le\hat P_y(r)$ as well as
\begin{align*}
    \hat P_y(r)=\sum_{n=0}^\infty p_n(y)r^n\le \frac{e^y}{1-\xi}\sum_{n=0}^\infty p_{n+1}(0)r^n=\frac{e^y}{r(1-\xi)}(\hat P_0(r)-1).
\end{align*}
Now let $y\le 0$, so $p_n(y)\le p_n(0)$ by monotonicity, and $\theta y\ge 0$ together with \eqref{proof_roc} yield
\begin{align*}
    p_{n+1}(y)&= \int_{(\theta y,\infty)}(1-\xi)  e^{-t}p_n(t)\d t
   \ge (1-\xi)p_n(0)\int_{(\theta y,\infty)}  e^{-t}\d t
   =(1-\xi)e^{-\theta y}p_n(0).
\end{align*}
For any $r>0$, we thus have $\hat P_y(r)\le \hat P_0(r)$ and 
\begin{align*}
     \hat P_y(r)-1=\sum_{n=0}^\infty p_{n+1}(y)r^{n+1}\ge (1-\xi)e^{-\theta y} r \hat P_0(r).
\end{align*}
In particular, for any $y\in\R$ and $r>0$, $\hat P_y(r)<\infty$ if and only if $\hat P_0(r)<\infty$, and since all coefficients are non-negative, the two power series have the same radius of convergence.
\end{proof} 
To determine the radius of convergence of $\hat{P}_{\xi,y,\theta}$, it suffices by the preceding lemma to consider the case $y=0$.  Lemma~\ref{gf_neg_theta} and Corollary~\ref{gf_theta-1} yield  
     \begin{align}\label{quotient_hat_P_thetaneg}
      \hat{P}_{\xi,0,\theta}(z)&=  
      \cos_{-\theta}(uz)
     -\gamma^{1/2}\frac{R^1(z)}{(1-\xi) z}\sin_{-\theta}(uz)\notag
     \\&=\cos_{-\theta}(uz)
     -\gamma^{1/2}\frac{\cos_{-\theta}(uz)-\gamma^{-1/2}\sin_{-\theta}(uz)}{\cos_{-\theta}(uz)+\gamma^{1/2}\sin_{-\theta}(uz)}\sin_{-\theta}(uz)\notag
     \\&=\frac{\cos_{-\theta}(uz)^2+\sin_{-\theta}(uz)^2}{\cos_{-\theta}(uz)+\gamma^{1/2}\sin_{-\theta}(uz)}.
    \end{align}
We denote the denominator by $$g(z)\coloneqq\cos_{-\theta}(uz)+\gamma^{1/2}\sin_{-\theta}(uz)=\cos_{-\theta}(-uz)-\gamma^{1/2}\sin_{-\theta}(-uz),$$
and set $z_0\coloneqq\frac{1}{\lambda}$, where we recall that $\lambda=\frac{-u}{\arctan_{-\theta}\left(\gamma^{-1/2}\right)}$.

The following lemmas provide the analytic ingredients needed to extract the coefficient asymptotics of $\hat{P}_{\xi,y,\theta}$. We first show that $z_0$ is the unique zero of $g$ in the closed disk $|z|\leq z_0$ and that this zero is simple. We then verify that all additional possible singularities arising from the representation of $\hat{P}_{\xi,y,\theta}$ are removable. These results identify $z_0$ as the radius of convergence and as the unique dominant simple pole. Finally, as in the other cases, Lemma~\ref{asymptotic_powerseries} yields the
persistence asymptotics once the corresponding residue and its
sign have been determined.
\begin{Lemma}\label{zeros_g_neg_theta}
    For $\theta<0$ and $\xi\in(0,1)$, the unique zero $z$ of $g(z)$ with $|z|\le z_0$ is given by $z=z_0$. 
\end{Lemma}
\begin{proof}
 In the case $\cos_{-\theta}(-uz)=0$, $\gamma\neq 0$ implies $\sin_{-\theta}(-uz)=0$ for any zero $z$ of $g$, a contradiction to Corollary~\ref{sin_and_cos_prelim_3}, so $\cos_{-\theta}(-uz)\neq 0$ for every zero $z$ of $g$. For real $0<z\le z_0$, we have that $ g(z)= \cos_{-\theta}(-uz)-\gamma^{1/2}\sin_{-\theta}(-uz)=0$ if and only if $\frac{\gamma^{1/2}\sin_{-\theta}(-uz)}{\cos_{-\theta}(-uz)}=1$. This is equivalent to $\tan_{-\theta}(-uz)=\gamma^{-1/2}$, which has the unique solution $z=z_0$ since $\tan_{-\theta}:[0,z_c)\to[0,\infty)$ is bijective and 
\begin{align*}
    0<-uz\le -uz_0=\arctan_{-\theta}(\gamma^{-1/2})<z_c.
\end{align*}
Now consider real $-z_0\leq z\leq 0$. We have $0\le uz\le-u{z_0}<z_c$ and $\cos_{-\theta}(0)=1$, implying $\cos_{-\theta}(-u{z})=\cos_{-\theta}(u{z})> 0$ since $z_c$ is the smallest positive real zero of $\cos_{-\theta}$. Moreover, since $\tan_{-\theta}$ is a non-negative function, we have $\sin_{-\theta}(uz)\ge0$ as well, so
\begin{align*}
    g(z)&=\cos_{-\theta}(-uz)-\gamma^{1/2}\sin_{-\theta}(-uz)
        =\cos_{-\theta}(uz)+\gamma^{1/2}\sin_{-\theta}(uz)>0.
\end{align*}
Finally, let $z\in\C\backslash \R$ and set $x\coloneqq -iuz$, so $\Re(x)\ne 0$. Then, $g(z)= \cos_{-\theta}(-uz)-\gamma^{1/2}\sin_{-\theta}(-uz)=0$ if and only if $ \frac{e_{-\theta}(x)+e_{-\theta}(-x)}{2}-\gamma^{1/2}\frac{e_{-\theta}(x)-e_{-\theta}(-x)}{2i}=0$. This is equivalent to $ e_{-\theta}(x)+e_{-\theta}(-x)+i\gamma^{1/2}(e_{-\theta}(x)-e_{-\theta}(-x))=0$ and thus to $e_{-\theta}(x)=-\frac{1-i\gamma^{1/2}}{1+i\gamma^{1/2}}e_{-\theta}(-x)$.
For $\theta<-1$, since $\left|\frac{1-i\gamma^{1/2}}{1+i\gamma^{1/2}}\right|=1$ and $\Re(x)\ne 0$, Lemma~\ref{estimate_e_theta_ge_1_prelim} yields that there is no such $z$ with  $e_{-\theta}(x)=-\frac{1-i\gamma^{1/2}}{1+i\gamma^{1/2}}e_{-\theta}(-x)$. The same argument holds for $-1<\theta<0$ if $\pm x$ is no pole of $e_{-\theta}$. If either $x$ or $-x$ is a pole of $e_{-\theta}$, then exactly one of the two terms in the representation of $g$ has a pole with non-zero coefficient, so $g$ has a non-removable singularity in $z$ and hence $z$ is not a zero of $g$. For $\theta=-1$, this follows from the identity $|e^x|=e^{\Re(x)}$. In conclusion, there is no non-real zero of $g$.
\end{proof}
\begin{Lemma}\label{z_0_simple}
   For $\theta<0$ and $\xi\in(0,1)$, $z_0$ is a simple zero of $g$, and the following constant is finite:
   \begin{align}\label{c_non_zero_negtheta}
       c(\theta,\lambda,\xi)\coloneqq \lim_{z\to z_0}\left(z_0-z\right)\frac{R^1(z)}{(1-\xi) z}=-\frac{\cos_{-\theta}(uz_0)-\gamma^{-1/2}\sin_{-\theta}(uz_0)}{g'(z_0)}\neq 0.
  \end{align}
\end{Lemma}
\begin{proof}
We show that $g'(z_0)\ne 0$. To do so, set $F(z)\coloneqq \cos_{-\theta}(z)-\gamma^{1/2}\sin_{-\theta}(z)$, so $g(z)=F(-uz)$, $g'(z_0)=-uF'(-uz_0)$ and $F(-uz_0)=0$. First, let $\theta\neq -1$. By Corollary~\ref{estimate_pos_e_prime_e}, for every real $z>0$, 
\begin{align*}
    \frac{e'_{-\theta} (iz)}{e_{-\theta} (iz)}+\frac{e'_{-\theta}(-iz)}{e_{-\theta}(-iz)}>0.
\end{align*}
In particular, this holds for $z=-uz_0$ since $u<0$.
Plugging in the representations of $\sin_{-\theta}$ and $\cos_{-\theta}$, we can rewrite
\begin{align*}
    F(z)=\frac{(1+i\gamma^{1/2})e_{-\theta} (iz)+(1-i\gamma^{1/2})e_{-\theta}(-iz)}{2}.
\end{align*}
Evaluating the derivative at $-uz_0$, we get
\begin{align*}
    F'(-uz_0)=\frac{i}{2}\left((1+i\gamma^{1/2})e'_{-\theta} (-iuz_0)-(1-i\gamma^{1/2})e'_{-\theta}(iuz_0)\right).
\end{align*}
The assumption $F(-uz_0)=0$ implies $(1+i\gamma^{1/2})e_{-\theta}(-iuz_0)=-(1-i\gamma^{1/2})e_{-\theta}(iuz_0)$. Since Corollary~\ref{estimate_pos_e_prime_e} implies $e_{-\theta}(\pm iuz_0)\neq0$, we get
\begin{align*}
    F'(-uz_0)
    &=\frac{i}{2}\left((1+i\gamma^{1/2})e_{-\theta}(-iuz_0)\frac{e'_{-\theta}(-iuz_0)}{e_{-\theta}(-iuz_0)}-(1-i\gamma^{1/2})e_{-\theta}(iuz_0)\frac{e'_{-\theta}(iuz_0)}{e_{-\theta}(iuz_0)}\right)
    \\&=-\frac{i}{2}(1-i\gamma^{1/2})e_{-\theta} (iuz_0)\left(\frac{e'_{-\theta}(-iuz_0)}{e_{-\theta}(-iuz_0)}+\frac{e'_{-\theta} (iuz_0)}{e_{-\theta} (iuz_0)}\right)\ne 0,
\end{align*}
which yields that $z_0$ is a simple zero of $g$. For $\theta=-1$, this follows directly from $-uz_0=\arctan(\gamma^{-1/2})\in\left(0,\frac{\pi}{2}\right)$, so that
\begin{align*}
    F'(-uz_0)=-\sin(-uz_0)-\gamma^{1/2}\cos(-uz_0)<0.
\end{align*}
The rest of the claim now follows since $\cos_{-\theta}(-uz_0)=\gamma^{1/2}\sin_{-\theta}(-uz_0)> 0$ and $z_0$ is a simple zero of $g$, so Taylor's formula applied to $\frac{R^1(z)}{(1-\xi) z}=\frac{\cos_{-\theta}(uz)-\gamma^{-1/2}\sin_{-\theta}(uz)}{g(z)}$ yields
\begin{align*}
      \lim_{z\to z_0}\left(z_0-z\right)\frac{R^1(z)}{(1-\xi) z}=-\frac{\cos_{-\theta}(uz_0)-\gamma^{-1/2}\sin_{-\theta}(uz_0)}{g'(z_0)}\neq 0. \tag*{\qedhere}
\end{align*}
\end{proof}
For the next lemma, recall the definition of $z_k^\pm$ given in \eqref{def_zk}. Since we work with $-\theta$ instead of $\theta$, set analogously, for $-1<\theta<0$,
\begin{align*}
       w_k^+\coloneqq \frac{i}{(1+\theta)(-\theta)^k}\quad\text{and}\quad w_k^-\coloneqq \frac{-i}{(1+\theta)(-\theta)^k}.
   \end{align*}
\begin{Lemma}\label{poles_hatP_theta}
      For $\theta<0$ and $\xi\in(0,1)$, all possible non-removable singularities of $\hat{P}_{\xi,y,\theta}(z)$ are zeros of $g$. 
\end{Lemma}
\begin{proof} 
The claim is trivial in the case $\theta\le-1$ since all occurring factors apart from the denominator $g$ in the representation of $\hat{P}_{\xi,y,\theta}(z)$ given in Lemma~\ref{gf_neg_theta} are entire by Corollary~\ref{sin_and_cos_prelim_2} and Corollary~\ref{sine_and_cose_prelim}. 

For $-1<\theta<0$, we identify every possible singularity of $\hat{P}_{\xi,y,\theta}(z)$ not determined by the zeros of $g$ and then show that they are removable singularities. Recall that by Lemma~\ref{gf_neg_theta}, the meromorphic continuation of $\hat{P}_{\xi,y,\theta}$ is given by
\begin{align*}
      \hat{P}_{\xi,y,\theta}(z)=  
      \begin{cases}
          \gamma^{-1/2}\sin_{-\theta}(uz,-y)+\frac{R^1(z)}{(1-\xi) z}\cos_{-\theta}(uz,-y),~~~&y>0,
          \\\cos_{-\theta}(uz,y)
     -\gamma^{1/2}\frac{R^1(z)}{(1-\xi) z}\sin_{-\theta}(uz,y),&y\le 0,
      \end{cases}
    \end{align*}
    where $\frac{R^1(z)}{(1-\xi) z}=\frac{\cos_{-\theta}(uz)-\gamma^{-1/2}\sin_{-\theta}(uz)}{g(z)}$
and the functions $\sin_{-\theta}(uz)$, $\cos_{-\theta}(uz)$, $\sin_{-\theta}(uz,\pm y)$ and $\cos_{-\theta}(uz,\pm y)$ may have poles at $z=w_k^\pm/u=\pm\frac{i}{u(1+\theta)(-\theta)^k}$. First, consider $z=w_k^+/u$. By Lemma~\ref{residue_sin_cos_prelim}, it holds
\begin{align*}
    \mathrm{Res}_{z=w_k^+/u}\sin_{-\theta}(uz)&= i\mathrm{Res}_{z=w_k^+/u}\cos_{-\theta}(uz)
\end{align*}
and analogously for $\sin_{-\theta}(uz,\pm y)$ and $\cos_{-\theta}(uz,\pm y)$. Therefore, we get
\begin{align*}
    \lim_{z\to w_k^+/u}\frac{R^1(z)}{(1-\xi) z}
    &=\frac{\mathrm{Res}_{z=w_k^+/u}\cos_{-\theta}(uz)-\gamma^{-1/2}\mathrm{Res}_{z=w_k^+/u}\sin_{-\theta}(uz)}{\mathrm{Res}_{z=w_k^+/u}\cos_{-\theta}(uz)+\gamma^{1/2}\mathrm{Res}_{z=w_k^+/u}\sin_{-\theta}(uz)}
    \\&=\frac{(1-i\gamma^{-1/2})\mathrm{Res}_{z=w_k^+/u}\cos_{-\theta}(uz)}{(1+i\gamma^{1/2})\mathrm{Res}_{z=w_k^+/u}\cos_{-\theta}(uz)}
    \\&=\frac{1-i\gamma^{-1/2}}{1+i\gamma^{1/2}}=-i\gamma^{-1/2},
\end{align*}
where the resulting denominator in the second step does not vanish at $z= w_k^+/u$: Indeed, by Corollary~\ref{sin_and_cos_prelim_1} and since $u\neq0$,
$w_k^+/u$ is a simple pole of
$z\mapsto\cos_{-\theta}(uz)$. Hence, $\mathrm{Res}_{z=w_k^+/u}\cos_{-\theta}(uz)\neq0$ and thus
\begin{align*}
     \mathrm{Res}_{z=w_k^+/u}g(z)=(1+i\gamma^{1/2} )\mathrm{Res}_{z=w_k^+/u}\cos_{-\theta}(uz)\ne 0.
\end{align*}
Since the numerator and the denominator of $\frac{R^1(z)}{(1-\xi) z}$ have at most simple poles at $z= w_k^+/u$, multiplication of both by $(z- w_k^+/u)$ results in a holomorphic extension in a neighborhood of $w_k^+/u$. 
Hence, the quotient has a removable singularity at $z= w_k^+/u$ and extends holomorphically to this point with value $-i\gamma^{-1/2}$. Thus, we can calculate for $y>0$,
\begin{align*}
   \mathrm{Res}_{z=w_k^+/u}\hat{P}_{\xi,y,\theta}(z)=\gamma^{-1/2}i\mathrm{Res}_{z=w_k^+/u}\cos_{-\theta}(uz,-y)-i\gamma^{-1/2}\mathrm{Res}_{z=w_k^+/u}\cos_{-\theta}(uz,-y)=0
\end{align*}
and for $y\le 0$,
\begin{align*}
   \mathrm{Res}_{z=w_k^+/u}\hat{P}_{\xi,y,\theta}(z)=\mathrm{Res}_{z=w_k^+/u}\cos_{-\theta}(uz,y)-\gamma^{1/2}(-i\gamma^{-1/2})i\mathrm{Res}_{z=w_k^+/u}\cos_{-\theta}(uz,y)=0.
\end{align*}
Since the possible singularity of $\hat{P}_{\xi,y,\theta}(z)$ at $z=w_k^+/u$ is at most simple by Lemma~\ref{residue_sin_cos_prelim} and its residue is zero, it is removable.

The arguments for $z=w_k^-/u$ are similar. To sum up, $\hat{P}_{\xi,y,\theta}(z)$ admits a holomorphic extension across the removable singularities $z=w_k^\pm/u$ and thus, all possible non-removable singularities of $\hat{P}_{\xi,y,\theta}(z)$ are given by zeros of $g$, which shows the claim.
\end{proof}
\begin{Lemma}\label{roc_hatP_theta_le-1}
     For $\theta<0$ and $\xi\in(0,1)$, the radius of convergence $\delta$ of $\hat{P}_{\xi,y,\theta}(z)$ is given by $z_0=\frac{1}{\lambda}$.
\end{Lemma}
\begin{proof}
    By Lemma~\ref{roc_P_not_depending_on_y}, it suffices to determine the radius of convergence of $\hat{P}_{\xi,0,\theta}(z)$ and, by Lemma~\ref{poles_hatP_theta}, all possible non-removable singularities of $\hat{P}_{\xi,0,\theta}(z)$ are given by zeros of $g$. 
   By Lemma~\ref{zeros_g_neg_theta}, $z_0$ is the smallest real zero of $g$, so we first show that $z_0$ is a simple pole of $\hat{P}_{\xi,0,\theta}(z)$.
 Lemma~\ref{z_0_simple} and \eqref{quotient_hat_P_thetaneg} yield that
\begin{align*}
    \lim_{z\to z_0}\left(z_0-z\right)\hat{P}_{\xi,0,\theta}(z)=-\frac{\cos_{-\theta}(uz_0)^2+\sin_{-\theta}(uz_0)^2}{g'(z_0)}
\end{align*}
 is finite and non-zero since $\cos_{-\theta}(uz_0)^2+\sin_{-\theta}(uz_0)^2\neq 0$ by Corollary~\ref{sin_and_cos_prelim_3}, so $z_0$ is a simple pole of $\hat{P}_{\xi,0,\theta}(z)$ and therefore, $\delta\leq z_0$. 

Now assume that $\delta <z_0$. 
By Lemma~\ref{zeros_g_neg_theta}, for all $z\in\C$ with $|z|=\delta$, we have $g(z)\neq 0$. Lemma~\ref{poles_hatP_theta} yields that $\hat{P}_{\xi,0,\theta}(z)$ has an analytic extension in a neighborhood of every point of the circle $|z|=\delta$. By compactness, finitely many of these neighborhoods cover the circle. Moreover, $\hat{P}_{\xi,0,\theta}(z)$ is analytic on the disk $|z|<\delta$, so there is some $\epsilon >0$ such that $\hat{P}_{\xi,0,\theta}(z)$ admits an analytic continuation to $|z|<\delta+\epsilon$, which is a contradiction to $\delta$ being the radius of convergence of $\hat{P}_{\xi,0,\theta}(z)$. Thus, $\delta=z_0$.
\end{proof}
For $y\in\R$, define
\begin{align}\label{def_eigenfunction_neg_theta}
    h(y)\coloneqq h_{\theta,\xi}(y)&\coloneqq
   \begin{cases}
       \cos_{-\theta}\left(u z_0,-y\right) ,~~&y>0,
        \\-\gamma^{1/2}\sin_{-\theta}\left(u z_0,y\right),~~&y\le 0.
    \end{cases}
\end{align}
Note that for $\theta=-1$, this function becomes $h(y)=\cos\left(u z_0e^{-y}\right)$, for $y>0$, and $h(y)=-\gamma^{1/2}\sin\left(u z_0e^y\right)$, for $y\le 0$.

\begin{Lemma}\label{h_pos_thetale-1}
     For $\theta<0$, $\xi\in(0,1)$ and every $y\in\R$, the only possible non-removable
singularity of $\hat P_{\xi,y,\theta}$ in the closed disk
$|z|\le z_0$ is $z_0$. In fact, $z_0$ is a simple pole. Moreover, $c(\theta,\lambda,\xi)>0$ for $c(\theta,\lambda,\xi)$ defined in Lemma~\ref{z_0_simple} and $h(x)>0$ for all $x\in\R$. 
\end{Lemma}
\begin{proof}
   Lemma~\ref{poles_hatP_theta} together with Lemma~\ref{zeros_g_neg_theta} yields the first part of the claim. Moreover, for any $y\in\R$, by Lemma~\ref{roc_hatP_theta_le-1}, $z_0$ is the radius of convergence of $\hat{P}_{\xi,y,\theta}(z)$, so there must be a non-removable singularity $z$ with $|z|=z_0$. By the first part of the claim, the only possible such point is $z_0$, hence, $z_0$ is a non-removable singularity. 
   
   By Lemma~\ref{z_0_simple}, $g$ has a simple zero at $z_0$ and all remaining factors are holomorphic at $z_0$, so the singularity of $\hat{P}_{\xi,y,\theta}(z)$ is a simple pole and therefore, $\lim_{z\nearrow z_0}(z_0-z)\hat{P}_{\xi,y,\theta}(z)\ne 0$. Therefore, the representation of $\hat{P}_{\xi,y,\theta}(z)$ from Lemma~\ref{gf_neg_theta} or Corollary~\ref{gf_theta-1} together with
   \begin{align}\label{eqn_proof_C_pos}
       0\ne \lim_{z\nearrow z_0}(z_0-z)\hat{P}_{\xi,y,\theta}(z)=c(\theta,\lambda,\xi)h(y)
   \end{align}
    and $c(\theta,\lambda,\xi)\ne 0$ from \eqref{c_non_zero_negtheta} implies that $h(y)\neq 0$ for all $y\in\R$.  
    Further, $h$ is continuous: $$\lim_{x\searrow 0}h(x)=\cos_{-\theta}\left(u z_0\right)=-\gamma^{1/2}\sin_{-\theta}\left(u z_0\right)=h(0)$$ follows from $g(z_0)=0$ and on the regions $x>0$, $x\le 0$, $h$ is continuous by definition of the respective functions, so $h$ has constant sign by the intermediate value theorem. 
    
    To show positivity, note that we have $0<-u{z_0}<z_c$, so $\cos_{-\theta}(-u{z_0})=\cos_{-\theta}(u{z_0})> 0$ since $z_c$ is the smallest positive real zero of $\cos_{-\theta}$ and $\cos_{-\theta}(0)=1$. With $g(z_0)=0$, this yields $\sin_{-\theta}(u{z_0})<0$ and thus, $h(0)>0$ which implies $h(x)>0$ for all $x\in\R$.
    
    The fact that $c(\theta,\lambda,\xi)>0$ now follows from \eqref{eqn_proof_C_pos}: The left-hand side is positive since $z_0>0$ and all coefficients of $\hat{P}_{\xi,y,\theta}(z)$ are non-negative. $h>0$ then yields $c(\theta,\lambda,\xi)>0$.
\end{proof}

\subsection{Persistence probabilities and proof of the main theorems}
Having identified the unique dominant simple pole and the sign of its residue, we can now extract the persistence asymptotics.
\begin{Lemma}\label{asymptotic_neg_theta}
  For $\theta<0$,  $\xi\in(0,1)$ and $y\in\R$, we have, as $n\to\infty$,
  \begin{align*}
       p_n^{\xi,y}(\theta)\sim c(\theta,\lambda,\xi)h(y)\lambda^{n+1}
  \end{align*}
  with $h$ defined in \eqref{def_eigenfunction_neg_theta} and $c(\theta,\lambda,\xi)$ defined in \eqref{c_non_zero_negtheta}.
\end{Lemma}
\begin{proof}
    First, let $\theta\neq -1$. The claim for $\theta=-1$ follows analogously with the representation from Corollary~\ref{gf_theta-1} instead of Lemma~\ref{gf_neg_theta}. To apply Lemma~\ref{asymptotic_powerseries}, we construct a function $f$ with
    \begin{align}
        \hat{P}_{\xi,y,\theta}(z)=\frac{f(z)}{\frac{1}{\lambda}-z}~~~~\text{for all}~|z|<\frac{1}{\lambda}\label{eq_Lemma_neg},
    \end{align}
   where $f$ is analytic on $|z|<z_1$ with $\frac{1}{\lambda}<z_1$ and $f\left(\frac{1}{\lambda}\right)\neq 0$. By Lemma~\ref{gf_neg_theta}, 
   \begin{align*}
        f(z)&=  \left(\frac{1}{\lambda}-z\right)\hat{P}_{\xi,y,\theta}(z)
        \\&=\left(\frac{1}{\lambda}-z\right)\cdot
      \begin{cases}
          \gamma^{-1/2}\sin_{-\theta}(uz,-y)+\frac{R^1(z)}{(1-\xi) z}\cos_{-\theta}(uz,-y),~~~&y>0,
          \\\cos_{-\theta}(uz,y)
     -\gamma^{1/2}\frac{R^1(z)}{(1-\xi) z}\sin_{-\theta}(uz,y),&y\le 0,
     \end{cases}
   \end{align*}
   is a suitable choice for \eqref{eq_Lemma_neg}. Lemma~\ref{roc_hatP_theta_le-1} together with Lemma~\ref{h_pos_thetale-1} yield that $\frac{1}{\lambda}$ is the unique singularity with smallest modulus of $\hat{P}_{\xi,y,\theta}(z)$ and it is simple and isolated. In particular, there is some $z_1>\frac{1}{\lambda}$ such that the function $f$ is analytic on $|z|<z_1$. By Lemma~\ref{z_0_simple}, we have
   \begin{align*}
       f\left(\frac{1}{\lambda}\right)=\begin{cases}
         c(\theta,\lambda,\xi)\cos_{-\theta}\left(\frac{u}{\lambda},-y\right),~~~&y>0,
          \\-\gamma^{1/2}c(\theta,\lambda,\xi)\sin_{-\theta}\left(\frac{u}{\lambda},y\right),&y\le 0.
     \end{cases}
   \end{align*}
 By Lemma~\ref{z_0_simple} and
  Lemma~\ref{h_pos_thetale-1}, $f\left(\frac{1}{\lambda}\right)\neq0$ and thus, $f$ satisfies the  conditions for Lemma \ref{asymptotic_powerseries}, which now yields $p_n^{\xi,y}(\theta)\sim c(\theta,\lambda,\xi)h(y)\lambda^{n+1}$.
\end{proof}
\begin{Corollary}\label{asymptotic_theta_neg_ohne_y}
      For $\theta<0$ and  $\xi\in(0,1)$, we have, as $n\to\infty$,
  \begin{align*}
       p_n^{\xi}(\theta)\sim c_2(\theta,\xi)\lambda^{n+1},
  \end{align*}
  where $c_2(\theta,\xi)\coloneqq c(\theta,\lambda,\xi)\lambda$ and $c(\theta,\lambda,\xi)$ is defined in \eqref{c_non_zero_negtheta}.
\end{Corollary}
\begin{proof}
For fixed $n\in\N$, we have $\lim_{y\to+\infty} p_n^{\xi,y}(\theta)= p_{n-1}^{\xi}(\theta)$. Hence, by the monotone convergence theorem and with $\hat P_{\xi,\theta}(z)\coloneqq \sum_{n=0}^\infty p_n^\xi(\theta) z^n$, for $0\leq z<z_0$,
\begin{align*}
    \lim_{y\to\infty}\hat P_{\xi,y,\theta}(z)=1+z\hat P_{\xi,\theta}(z).
\end{align*}
We next determine the limit on the left-hand side. By Lemma~\ref{gf_neg_theta} if $\theta\neq -1$ and Corollary~\ref{gf_theta-1} if $\theta=-1$, we have
\begin{align*}
    \hat P_{\xi,y,\theta}(z)=\gamma^{-1/2}\sin_{-\theta}(uz,-y)+\frac{R^1(z)}{(1-\xi)z}\cos_{-\theta}(uz,-y).
\end{align*}
For every fixed real $z$, we have $\lim_{y\to\infty}\sin_{-\theta}(uz,-y)=0$ and $\lim_{y\to\infty}\cos_{-\theta}(uz,-y)=1$:
Indeed, if $\theta<-1$, these limits follow from the defining series,
dominated convergence, and Lemma~\ref{e_theta_entire_thetage1}.
For $\theta=-1$, they follow immediately from the continuous
extensions.

Finally, let $-1<\theta<0$ , so by Lemma~\ref{sin_e_and_cos_e_prelim},
\begin{align*}
    \sin_{-\theta}(uz,-y)
    &=e^{\frac{\theta y}{1+\theta}}\sum_{m=0}^\infty\frac{\left(\frac{-\theta y}{1+\theta}\right)^m}{m!}\sin_{-\theta}\bigl((-\theta)^m uz\bigr),\\
    \cos_{-\theta}(uz,-y)
    &= e^{\frac{\theta y}{1+\theta}}\sum_{m=0}^\infty\frac{\left(\frac{-\theta y}{1+\theta}\right)^m}{m!}\cos_{-\theta}\bigl((-\theta)^m uz\bigr).
\end{align*}
Since $(-\theta)^m uz\to0$, the terms in the two sums converge to
$0$ and $1$, respectively, and both sequences are bounded. Moreover,
the non-negative weights $e^{\frac{\theta y}{1+\theta}}\frac{\left(\frac{-\theta y}{1+\theta}\right)^m}{m!}$, $m\in\N_0$ sum to one, while their sum over any fixed finite set of indices
converges to zero as $y\to\infty$. Splitting the sums into finitely
many initial terms and the remaining tail therefore proves the two
limits.

Hence, for $0\le z<z_0$, we obtain
\begin{align*}
    1+z\hat P_{\xi,\theta}(z)=\lim_{y\to+\infty}\hat P_{\xi,y,\theta}(z)=\frac{R^1(z)}{(1-\xi)z},
\end{align*}
so the Identity Theorem yields that $\hat P_{\xi,\theta}(z)=\frac{R^1(z)-(1-\xi)z}{(1-\xi)z^2}$ on $|z|<z_0$ with a removable singularity at $z=0$. Recalling that, analogous to the proof of Lemma~\ref{asymptotic_neg_theta}, $f(z)\coloneqq (z_0-z)\hat P_{\xi,\theta}(z)$ is analytic on $|z|<z_1$ with some $z_1>z_0$ and
$\lim_{z\to z_0}f(z)=\frac{c(\theta,\lambda,\xi)}{z_0}$, Lemma~\ref{asymptotic_powerseries} yields $ p_n^{\xi}(\theta)\sim \frac{c(\theta,\lambda,\xi)}{z_0^2}z_0^{-n}$.
   Replacing $z_0=\frac{1}{\lambda}$ finishes the proof.
\end{proof}
\begin{proof}[Proof of Theorem~\ref{maintheorem} (iii)]
Applying the asymptotics from Lemma~\ref{asymptotic_neg_theta} and Corollary~\ref{asymptotic_theta_neg_ohne_y} to Lemma~\ref{eigenfunction_lemma} and Lemma~\ref{Kernel_lim_process} yields the claim.
\end{proof}
\begin{proof}[Proof of Theorem~\ref{invariant_distribution} (iii)]
Let $\theta<0$ and $\xi\in(0,1)$ and set $\bar\theta\coloneqq \theta^{-1}<0$. Define 
\begin{align*} 
\tilde h_{\theta,\xi}(x) \coloneqq h_{\bar\theta,\xi}(x), \qquad x\in\R,
\end{align*}
where $h_{\bar\theta,\xi}$ is the right eigenfunction satisfying \eqref{eq_eigenfunction} from Theorem~\ref{maintheorem} (iii).
Since $\theta<0$, it holds $ y>\theta x $ if and only if $ x>\bar\theta y $. Consequently, 
\begin{align*} 
\int_\R \mathds{1}_S(y,x) \tilde h_{\theta,\xi}(x) \phi(x)\d x 
= \int_\R \mathds{1}_{\{x>\bar\theta y\}} h_{\bar\theta,\xi}(x) \phi(x)\d x
= \lambda_{\bar\theta ,\xi} h_{\bar\theta ,\xi}(y). 
\end{align*} 
We now verify that $ \lambda_{\bar\theta ,\xi} = \lambda_{\theta,\xi}$. By Corollary~\ref{arctan_cor}, $ \arctan_{-\theta}(\gamma^{-1/2}) = \arctan_{-\bar\theta}(\gamma^{-1/2})$ and hence, $ \lambda_{\bar\theta ,\xi} = \frac{ \sqrt{\xi(1-\xi)} }{ \arctan_{-\theta}(\gamma^{-1/2}) }=\lambda_{\theta,\xi}$. We have thus proved the left eigenfunction equation 
\begin{align*} 
\int_\R \mathds{1}_S(y,x) \tilde h_{\theta,\xi}(x) \phi(x)\d x = \lambda_{\theta,\xi} \tilde h_{\theta,\xi}(y), \qquad y\in\R. 
\end{align*} 
Set 
\begin{align*}
I  \coloneqq \int_\R h_{\theta,\xi}(x) h_{\bar\theta ,\xi}(x) \phi(x)\d x.
\end{align*} 
Both right eigenfunctions are strictly positive on $\R$, and therefore $ I >0$. Moreover, Lemma~\ref{estimate_Px}, together with the persistence asymptotics from Lemma~\ref{asymptotic_neg_theta} and Corollary~\ref{asymptotic_theta_neg_ohne_y}, implies that both right eigenfunctions are bounded. Since $\phi$ is a probability density, $ I <\infty$. Lemma~\ref{invariant_distribution_limiting_process} now shows that the following probability measure is invariant for $p'$:
\begin{align*} 
\pi (\d x) = \frac{ h_{\theta,\xi}(x) h_{\bar\theta ,\xi}(x) }{I} \phi(x)\d x, \quad x\in\R.
\end{align*}

It is left to derive the explicit expression for this invariant distribution. By Corollary~\ref{arctan_cor}, we may abbreviate $a\coloneqq \arctan_{ -\theta}(\gamma^{-1/2})=\arctan_{-\bar\theta}(\gamma^{-1/2})$. Using the eigenfunctions from Theorem~\ref{maintheorem} (iii) and the explicit form of $\phi$,
it follows that 
\begin{align*}
    h_{\theta,\xi}(x) h_{\bar\theta,\xi}(x) \phi(x)\d x = (1-\xi) \begin{cases} e^{-x} \cos_{-\theta}(a,-x) \cos_{-\bar \theta}(a,-x)\d x, &x>0,\\ e^x \sin_{-\theta}(a,x) \sin_{-\bar \theta}(a,x)\d x, &x\leq0. \end{cases}
\end{align*}
Thus, the invariant distribution admits the explicit representation 
\begin{align*}
    \pi(\d x) = \frac{1}{b} 
    \begin{cases} e^{-x} \cos_{-\theta}(a,-x) \cos_{-\bar \theta}(a,-x)\d x, &x>0,\\ e^x \sin_{-\theta}(a,x) \sin_{-\bar \theta}(a,x)\d x, &x\leq0, \end{cases}
\end{align*}
with $b\coloneqq\frac{I}{1-\xi}= \int_0^\infty e^{-x} \cos_{-\theta}(a,-x) \cos_{-\bar \theta}(a,-x)\d x + \int_{-\infty}^0 e^x \sin_{-\theta}(a,x) \sin_{-\bar \theta}(a,x)\d x$.

To prove uniqueness, note that $p'$ has a transition density which is strictly positive whenever $y>\theta x$. Let $p^n(x,\cdot)$ denote the $n$-step transition density. For $x,y\in\R$ and $z>\max\{\theta x,\bar\theta y\}$, we have $z>\theta x$ as well as $y>\theta z$. Consequently, $p^2(x,\d y)=\int_{\R}p'(x,\d z)p'(z,\d y)$ has a strictly positive Lebesgue density for every $x\in\R$. Thus, $p'$ is irreducible, so it admits at most one invariant probability measure. Hence, $\pi$ is the unique invariant distribution.
\end{proof}

\section{The trivial cases}\label{trivialcases}
\paragraph{The case \texorpdfstring{$\theta=0$}{theta =0}.}
The case (v) in Theorem~\ref{maintheorem} and Theorem~\ref{invariant_distribution} is trivial: Since the decoupling parameter is zero, we condition the i.i.d.\ innovations $(X_i)$ on positivity giving $p_n^{\xi,y}(0)=(1-\xi)^n$ and so $\lambda=1-\xi$ and $h\equiv 1$, and the resulting conditioned chain consists of i.i.d.\ standard exponential random variables (unless $\xi=1$ where this is impossible).

\paragraph{The case \texorpdfstring{$\theta>0$}{theta >0}.}
 In the special case $0<\theta<1$ and $\xi=1$, we have $X_i<0$ a.s., and $X_{i+1} >\theta X_i$ forces the random variables to be closer and closer to $0$. Correspondingly, one can show (with Lemma~\ref{gf_pos_theta} and Lemma~\ref{singularities_pos_theta} for $y<0$) that $(p_n^{1,y}(\theta))$ decays superexponentially, while trivially $p_n^{1,y}(\theta)=0$ for $y\geq 0$. Thus, this case does not fall within the positive-exponential-rate Doob-transform framework.

Similarly, for $\theta>1$ and $\xi=0$, the random variables are a.s.\ non-negative and $X_{i+1}> \theta X_i$ forces them to become larger and larger. Correspondingly, $\left(p_n^{0,y}(\theta)\right)_{n\in\N_0}$ have a superexponential decay, which can be shown by Lemma~\ref{asymptotic_pos_theta_ge1_pos_y} (for $y>0$) and Lemma~\ref{gf_pos_theta} and Lemma~\ref{e_theta_entire_thetage1} (for $y\leq 0$). So, again this case does not fall within the positive-exponential-rate Doob-transform framework. 

\paragraph{The case \texorpdfstring{$\theta<0$}{theta <0}.}
In the special case $\xi=1$, the $X_i$ are negative, for $i\geq 2$, so that $X_{i+1}>\theta X_i$ is impossible for a negative $\theta$. Thus, the question does not make sense here, as $p_n^{1,y}(\theta)=0$, $n\geq 2$.

Finally, let us treat the case (iv) in Theorem~\ref{maintheorem} and Theorem~\ref{invariant_distribution}, i.e.\ $\xi=0$. Here, $X_i>0$ and with $\theta<0$, the constraints $X_{i+1}>\theta X_i$ are always satisfied (unless the starting point is negative). 
In terms of $h$ and $\lambda$, one can check readily that $\lambda=1$ and $h$ is given by (\ref{eqn:lasth}).

\end{document}